\documentclass[twoside,12pt]{article} 
\usepackage[margin=1.0in]{geometry}

\usepackage{amssymb,amsmath,amsthm}
\usepackage{mathtools}
\usepackage{tasks}
\usepackage{bm,lscape}
\usepackage{bbm}
\usepackage{mathrsfs,dsfont}
\usepackage{wasysym}
\RequirePackage[
    colorlinks=true,
    linkcolor=blue,
    urlcolor=blue,
    citecolor=blue]{hyperref}
\RequirePackage{natbib}
\usepackage[utf8]{inputenc}
\usepackage[english]{babel}
\usepackage{algorithm}
\usepackage{algpseudocode}

\usepackage{graphicx}
\usepackage{color}
\usepackage{rotating}
\usepackage{authblk}
\usepackage{appendix}

\allowdisplaybreaks

\newcommand{\dR}{\mathbb{R}}
\newcommand{\dC}{\mathbb{C}}
\newcommand{\dN}{\mathbb{N}}
\newcommand{\dP}{\mathbb{P}}
\newcommand{\cB}{\mathcal{B}}
\newcommand{\cM}{\mathcal{M}}
\newcommand{\cF}{\mathcal{F}}
\newcommand{\cP}{\mathcal{P}}
\newcommand{\cL}{\mathcal{L}}
\newcommand{\cG}{\mathcal{G}}
\newcommand{\veps}{\varepsilon}

\newcommand{\rI}{\mathrm{I}}

\allowdisplaybreaks

\newcommand{\E}{\mathds{E}}

\newcommand{\cN}{\mathcal{N}}

\newcommand{\dL}{\mathbb{L}}

\newtheorem{thm}{Theorem}[section]
\newtheorem{lem}[thm]{Lemma}

\newtheorem{cor}[thm]{Corollary}
\newtheorem{remark}[thm]{Remark}

\providecommand{\keywords}[1]
{
  \small	
  \textbf{\textit{Keywords:}} #1
}

\begin{document}

\title{A Gaussian process limit for the self-normalized Ewens-Pitman process}


\author[1]{Bernard Bercu\thanks{bernard.bercu@math.u-bordeaux.fr}}
\author[2]{Stefano Favaro\thanks{stefano.favaro@unito.it}}
\affil[1]{\small{Institut de Math\'ematiques de Bordeaux, Universit\'e de Bordeaux, France}}
\affil[2]{\small{Department of Economics and Statistics, University of Torino and Collegio Carlo Alberto, Italy}}

\maketitle

\begin{abstract}
For an integer $n\geq1$, consider a random partition $\Pi_{n}$ of $\{1,\ldots,n\}$ into $K_{n}$ partition sets with $K_{r,n}$ partition subsets of size $r=1,\ldots,n$, and assume $\Pi_{n}$ distributed according to the Ewens-Pitman model with parameters $\alpha\in]0,1[$ and $\theta>-\alpha$. Although the large-$n$ asymptotic behaviors of $K_{n}$ and $K_{r,n}$ are well understood in terms of almost sure convergence and Gaussian fluctuations, much less is known about the asymptotic behavior of $P_{r,n}=K_{r,n}/K_n$ and of the self-normalized Ewens-Pitman process $(P_{1,n},P_{2,n},\dots)$. Motivated by the almost sure convergence of $(P_{1,n},P_{2,n},\dots)$ to the Sibuya distribution $p_{\alpha}=(p_{\alpha}(1),p_{\alpha}(2),\ldots)$, where $p_{\alpha}(r)$ is the probability mass at $r=1,2,\ldots$, we establish the $\ell^{2}$ distributional convergence
\begin{displaymath}
\sqrt{K_{n}}((P_{1,n},\,P_{2,n},\ldots)-p_{\alpha})\underset{n\rightarrow+\infty}{\overset{\cL}{\longrightarrow}}\mathcal{G}(\Gamma_\alpha),
\end{displaymath}
where $\mathcal{G}(\Gamma_\alpha)$ stands for a centered Gaussian process with covariance matrix $\Gamma_\alpha=diag(p_{\alpha}) - p_{\alpha} p_{\alpha}^T$. We apply our result to the estimation of the parameter
$\alpha$.

\end{abstract}

\keywords{Ewens-Pitman model; exchangeable random partition; Gaussian process; martingale; Sibuya distribution}


\section{Introduction}

The Ewens-Pitman model was introduced in \citet{Pit(95)} as a generalization of the celebrated Ewens model in population genetics \citep{Ewe(72)}. For an integer $n\geq 1$, consider a random partition of $[n]=\{1,\ldots,n\}$ into $K_{n}\in\{1,\ldots,n\}$ partition sets, and, for $r=1,\ldots,n$, denote by $K_{r,n}\in\{0,1,\ldots,n\}$ the number of partitions subsets of size $r$, i.e.,
$$
n=\sum_{r=1}^{n}rK_{r,n}\qquad \text{and}\qquad K_{n}=\sum_{r=1}^{n}K_{r,n}.
$$ 
For $\alpha\in[0,1[$ and $\theta>-\alpha$, the Ewens-Pitman model assigns to $\mathbf{K}_{n}=(K_{1,n},\ldots,K_{n,n})$ the probability
\begin{equation}\label{epsm}
\mathbb{P}(\mathbf{K}_{n}=(k_{1},\ldots,k_{n}))=n!\frac{\left(\frac{\theta}{\alpha}\right)^{(s_n)}}{(\theta)^{(n)}}\prod_{i=1}^{n}\left(\frac{\alpha(1-\alpha)^{(i-1)}}{i!}\right)^{k_{i}}\frac{1}{k_{i}!},
\end{equation}
where $s_n=k_1+\cdots+k_n$, and, for any $a\in \dR$, $(a)^{(n)}$ stands for the rising factorial of $a$ of order $n$, that is $(a)^{(n)}=a(a+1)\cdots(a+n-1)$. The Ewens model corresponds to \eqref{epsm} with $\alpha=0$. The probability distribution \eqref{epsm} admits a sequential or generative construction in terms of Chinese restaurant process \citep{Pit(95),Fen(98)}, a Poisson process construction by random sampling from the Pitman-Yor random measure \citep{Per(92),Pit(97)}, and a construction through compound Poisson (negative-Binomial) random partitions \cite{Dol(20),Dol(21)}. We refer to \citet[Chapter 3 and Chapter 4]{Pit(06)} for an overview of the Ewens-Pitman model, including applications in population genetics, excursion theory, Bayesian nonparametric statistics, combinatorics and statistical physics.

For $\alpha=0$ and $\theta>0$ there have been several works investigating the large-$n$ asymptotic behaviour of $K_{n}$ and the $K_{r,n}$. In particular, by exploiting the fact that $K_{n}$ is a sum of $n$ independent Bernoulli random variables, \citet[Theorem 2.3]{Kor(73)} showed that
\begin{equation}\label{eq:k0}
\lim_{n\rightarrow+\infty}\frac{K_{n}}{\log n}=\theta\qquad\text{a.s.}
\end{equation}
A central limit theorem for $K_{n}$ follows directly from Lindeberg-L\'evy theorem, along with a Berry-Esseen inequality. As regard the $K_{r,n}$'s, from \citet[Theorem 1]{Arr(92)} it follows that for all $r\geq1$
\begin{equation}\label{eq:kr0}
K_{r,n}\underset{n\rightarrow+\infty}{\overset{\cL}{\longrightarrow}}Z_{r}
\end{equation}
and
\begin{equation}\label{eq:kr0_all}
(K_{1,n},\,K_{2,n},\ldots)\underset{n\rightarrow+\infty}{\overset{\cL}{\longrightarrow}}(Z_{1},\,Z_{2},\ldots),
\end{equation}
where the $Z_{r}$'s are independent random variables with Poisson $\cP(\theta/r)$ distribution. \citet{Bar(92)} established a (quantitative) refinement  of \eqref{eq:kr0_all}  by providing tight lower and upper bounds for the Wasserstein distance between the laws of the Ewens process $(K_{1,n},\,K_{2,n},\ldots)$ and the Poisson process $(Z_{r})$. We refer \citet[Chapter 4]{Arr(03)} for a comprehensive overview.

For $\alpha\in]0,1[$ and $\alpha+ \theta>0$, results analogous to \eqref{eq:k0}-\eqref{eq:kr0_all} are available in the literature. In particular, by exploiting a martingale construction for $K_{n}$, \citet[Theorem 3.8]{Pit(06)} showed that
\begin{equation}\label{eq:kalpha}
\lim_{n\rightarrow+\infty}\frac{K_{n}}{n^{\alpha}}=S_{\alpha,\theta}\qquad\text{a.s.},
\end{equation}
where $S_{\alpha,\theta}$ is a positive and almost surely finite random variable, typically referred to as $\alpha$-diversity \citep{Pit(03)}. As regard the $K_{r,n}$'s, from \citet[Lemma 1.1]{Pit(06)} it follows that for all $r\geq1$
\begin{equation}\label{eq:kralpha}
\lim_{n\rightarrow+\infty}\frac{K_{r,n}}{n^{\alpha}}=p_{\alpha}(r) S_{\alpha,\theta}\qquad\text{a.s.},
\end{equation}
where
\begin{equation}
\label{Sibuya}
p_{\alpha}(r)=\frac{\alpha(1-\alpha)^{(r-1)}}{r!}.
\end{equation}
We refer to \citet{Ber(24)} for an alternative martingale approach to the limits \eqref{eq:kalpha} and \eqref{eq:kralpha}, also including convergence in $\mathbb{L}^{p}$, Gaussian fluctuations and laws of iterated logarithm.

\begin{remark}
Beyond almost sure and Gaussian fluctuations, $K_{n}$ and $K_{r,n}$ have been investigated with respect to large deviations and moderate deviations \citep{Fen(98),Fav(14),Fav(18),BF25}. Some non-asymptotic results for $K_{n}$ have been also established in terms of Berry-Esseen inequalities \citep{Dol(20),Dol(21)} and concentration inequalities \citep{Per(22),BF25}.
\end{remark}

\subsection{Main result}

For $\alpha\in]0,1[$ and $\alpha+ \theta>0$, the large-$n$ asymptotic behaviour of the Ewens-Pitman process $(K_{r,n},\,K_{2,n},\ldots)$ follows from the almost sure convergence \eqref{eq:kralpha}, that is
\begin{equation}\label{eq:kralpha_all}
\lim_{n\rightarrow+\infty}\left(\frac{K_{1,n}}{n^{\alpha}},\,\frac{K_{2,n}}{n^{\alpha}},\ldots\right)=(p_{\alpha}(1)S_{\alpha,\theta},\,p_{\alpha}(2)S_{\alpha,\theta},\ldots)\qquad\text{a.s.}
\end{equation}
which extends \eqref{eq:kr0_all}. The weights $p_\alpha=(p_{\alpha}(1),p_{\alpha}(2),\ldots)$ in \eqref{eq:kralpha_all} form a probability distribution on $\mathbb{N}$, which is known as Sibuya distribution \citep{Dev(93)}, namely $p_{\alpha}(r)\in]0,1[$ for all $r\geq1$ and 
\begin{equation}\label{sum_sibuya}
\sum_{r=1}^\infty p_{\alpha}(r)=\sum_{r=1}^\infty \frac{\alpha(1-\alpha)^{(r-1)}}{r!} =1.
\end{equation}
In contrast to the case $\alpha=0$, for $\alpha\in]0,1[$ there exists an interplay between the almost sure convergences \eqref{eq:kalpha} and \eqref{eq:kralpha_all}, which underpins the interpretation of $p_{\alpha}$ as the large-$n$ limiting proportions of partitions subsets \citep[Lemma 1.1]{Pit(06)}. It motivates us to study the asymptotic behavior of the self-normalized Ewens-Pitman process $(P_{1,n},\,P_{2,n},\ldots)$ defined, for all $r\geq1$, by
\begin{equation}
\label{DEFPrn}
P_{r,n}=\frac{K_{r,n}}{K_{n}}.
\end{equation}
An almost sure convergence of $(P_{1,n},\,P_{2,n},\ldots)$ follows from the almost sure convergences \eqref{eq:kalpha} and \eqref{eq:kralpha_all}, that is 
\begin{equation}\label{eq:kralpha_self}
\lim_{n\rightarrow+\infty}(P_{1,n},\,P_{2,n},\ldots)=p_{\alpha}\qquad\text{a.s.}
\end{equation}
The next theorem is the main result of the paper, proving a Gaussian process limit for $(P_{1,n},\,P_{2,n},\ldots)$.

\begin{thm}\label{main_teo}
Assume that $\alpha\in]0,1[$ and $\alpha+\theta>0$. Then, we have the $\ell^{2}$ distributional convergence
\begin{equation}\label{eq:gaussproc}
\sqrt{K_{n}}((P_{1,n},\,P_{2,n},\ldots)-p_{\alpha})\underset{n\rightarrow+\infty}{\overset{\cL}{\longrightarrow}}\mathcal{G}(\Gamma_\alpha),
\end{equation}
where $\mathcal{G}(\Gamma_\alpha)$ stands for a centered Gaussian process with infinite-dimensional covariance matrix 
\begin{equation}
\label{DEFGamma}
\Gamma_\alpha=diag(p_{\alpha}) - p_{\alpha} p_{\alpha}^T.
\end{equation} 
\end{thm}

It follows from \eqref{eq:kralpha_self} that the self-normalized Ewens-Pitman process $(P_{1,n},\,P_{2,r},\ldots)$ is a strongly consistent estimator of Sibuya distribution $p_{\alpha}$. Theorem \ref{main_teo} provides an effective tool to quantify uncertainty of this estimator by constructing a large-$n$ asymptotic confidence ball for $p_{\alpha}$.

\subsection{Estimation of the parameter $\alpha$}

A noteworthy special case of Theorem \ref{main_teo} is the case $r=1$, which connects to the estimation of the parameter $\alpha\in]0,1[$ in the Ewens-Pitman model. Several works have focused on the maximum likelihood estimator of $\alpha$, by relying on the explicit distribution \eqref{epsm}, and provided theoretical guarantees of the estimator in terms of consistency and Gaussian limits \citep{Fra(22),FN(21),Bal(24)}. However, the maximum likelihood estimator of $\alpha$ is not available in closed form and must be obtained numerically via iterative optimization routines, which are computationally demanding and sensitive to initialization. Instead, here, we propose to estimate $\alpha$ by
\begin{equation*}
\widehat{\alpha}_{n}=P_{1,n}=\frac{K_{1,n}}{K_{n}}.
\end{equation*}
The estimator $\widehat{\alpha}_{n}$ enjoys both practical and theoretical advantages. First, it is available in closed form, requiring only the computation of $K_{1,n}$ and $K_n$, in contrast with the maximum likelihood estimator. Second, it is computationally trivial and scales naturally to large sample sizes. Third, its asymptotic theoretical guarantees follows from the Ewens-Pitman model, and do not rely on approximations of the likelihood function. The next corollary follows from Equation \eqref{eq:kralpha_all} and Theorem \ref{main_teo}, providing theoretical guarantees of $\widehat{\alpha}_{n}$ as an estimator of the parameter $\alpha$.
\begin{cor}
\label{main_cor}
Assume that $\alpha\in]0,1[$ and  $\alpha+\theta>0$. Then, we have the almost sure convergence
\begin{equation}
\label{ASCVGALPHA}
\lim_{n\rightarrow+\infty}\widehat{\alpha}_{n}=\alpha\qquad\text{a.s.}
\end{equation}
Moreover, we also have the asymptotic normality
\begin{equation}
\label{WCVGALPHA}
\sqrt{K_n}( \widehat{\alpha}_{n} -\alpha )
\underset{n\rightarrow+\infty}{\overset{\cL}{\longrightarrow}}
\mathscr{N}(0, \alpha(1-\alpha)).
\end{equation} 
\end{cor}

Corollary \ref{main_cor} allows us to provide a large-$n$ asymptotic confidence interval for the parameter $\alpha$. More precisely, for any prescribed confidence level $1-\gamma\in]0,1[$, let $z_{1-\gamma/2}$ stands for the $(1-\gamma/2)$-quantile of the $\cN(0,1)$ distribution. Then an asymptotic $(1-\gamma)$-confidence interval for $\alpha$ is given by
\begin{equation}\label{CI_alpha}
\left[\widehat{\alpha}_{n} - z_{1-\gamma/2}\sqrt{\frac{\widehat{\alpha}_{n}(1-\widehat{\alpha}_{n})}{K_n}},\,\widehat{\alpha}_{n} + z_{1-\gamma/2}\sqrt{\frac{\widehat{\alpha}_{n}(1-\widehat{\alpha}_{n})}{K_n}}
\right].
\end{equation}

\subsection{Organization of the paper}

The paper is structured as follows. Section \ref{sec2} is devoted to the proof of Theorem \ref{main_teo}. In Section \ref{sec3}, we discuss some directions of future research on the self-normalized Ewens-Pitman process. The appendix is dedicated to the sequential construction of the Ewens-Pitman model and to moment formulae for the random variables $K_{n}$ and $(K_{n}, K_{r,n})$.


\section{Proof of Theorem \ref{main_teo}}\label{sec2}
The proof of Theorem \ref{main_teo} relies on a tailor-made martingale construction for the $K_{r,n}$'s, in combination with a non-standard central limit theorem for multi-dimensional martingales \citep{Tou(91),Tou(94)}, see also \citet{Ber(21)} for details.

\subsection{The martingale construction} 
\label{Sub-MC}


From the sequential or generative construction of the Ewens-Pitman model \citep[Proposition 9]{Pit(95)}, we have for all $n \geq 1$,
\begin{equation}\label{DEFK}
K_{n+1} = K_n + \xi_{n+1},
\end{equation}
where the conditional distribution of $\xi_{n+1}$ given the $\sigma$-algebra $\cF_n=\sigma(K_1, \ldots,K_n)$ is the Bernouilli $\cB(p_n)$ distribution with parameter
\begin{equation}
\label{DEFPN}
p_n=\frac{\alpha K_n+ \theta}{n+\theta}.
\end{equation} 
We refer the reader to Appendix \ref{appa} for more details on \eqref{DEFPN}. Let $(\mathcal{F}_{r,n})$ be the sequence of $\sigma$-algebras given by $\mathcal{F}_{1,n}=\sigma(K_1, \ldots,K_n,K_{1,1},\ldots,K_{1,n})$ and, for all $r\geq2$, $\mathcal{F}_{r,n}=\mathcal{F}_{1,n}\cup\mathcal{G}_{2,n}\cup\cdots\cup\mathcal{G}_{r,n}$ where $\mathcal{G}_{r,n}=\sigma(K_{r,1},\ldots,K_{r,n})$. Again, by relying on the sequential or generative construction of the  Ewens-Pitman model \citep{Ber(24)}, we have for all $r \geq 2$ and for all $n \geq 1$,
\begin{equation}\label{DEFKr}
K_{r,n+1}=K_{r,n}+\xi_{r,n+1},
\end{equation}
where $\xi_{r,n+1}$ given the $\sigma$-algebra $\mathcal{F}_{r,n}$ is distributed according to the following probability distribution
\begin{equation}
\label{DEFXIKr}
   \dP(\xi_{r,n+1}=k\,|\,\mathcal{F}_{r,n})=\left \{ \begin{array}{ccc}
    {p_{r,n}}  & \text{ if } & \ \ k=1 \vspace{1ex}\\
    {q_{r,n}}  & \text{ if } & \ \ k=-1 \vspace{1ex}\\
     {1-p_{r,n}-q_{r,n}} & \text{ if }  & \ \ k=0,
   \end{array}  \right.
\end{equation}
with
\begin{equation}
\label{DEFIPr}
p_{r,n}=\frac{(r-1-\alpha)K_{r-1,n}}{n+\theta}\qquad \text{and}\qquad q_{r,n}=\frac{(r-\alpha)K_{r,n}}{n+\theta}.
\end{equation}
We refer the reader to Appendix \ref{appa} for more details on \eqref{DEFIPr}. Decomposition \eqref{DEFKr} holds for $r=1$ by taking $p_{1,n}=p_n$. One can also observe that for all $r \geq 2$, $p_{r,n}=q_{r-1,n}$. 
From now on, let $(a_{r,n})$ be the sequence defined by $a_{r,n}=1$ for all $1 \leq n \leq r$ and for all 
$n\geq r+1$,
\begin{equation}
\label{DEFAr}
a_{r,n}=\prod_{k=r}^{n-1}\left(\frac{k+\theta}{k+\alpha+\theta-r}\right)=\frac{(\theta+1)^{(n-1)}}{(\alpha+\theta-r+1)^{(n-1)}}.
\end{equation}
Moreover, denote by $(M_{r,n})$ the sequence of random variables given, for all $r \geq 1$, by
\begin{equation}
\label{DEFMr}
M_{r,n}=a_{r,n}\left(S_{r,n}+ \frac{(-1)^{r} p_{\alpha}(r) \theta}{\alpha}\right)
\end{equation}
with
\begin{equation}
\label{DEFSNr}
S_{r,n}= \sum_{i=1}^{r}b_{r,i}K_{i,n}+(-1)^r p_{\alpha}(r) K_{n},
\end{equation}
where for all $i=1,\ldots,r$, 
\begin{equation}
\label{DEFbr}
b_{r,i}=(-1)^{r-i}\frac{(i-\alpha)^{(r-i)}}{(r-i)!}.
\end{equation}

\begin{lem}
\label{L-KEY}
For all $r \geq 1$, $(M_{r,n})$ is a locally square integrable martingale with predictable quadratic variation $\langle M_r \rangle_{n}$ satisfying
\begin{equation}
\label{LIMIPINMr}
\lim_{n\rightarrow+\infty}\frac{\langle M_{r}\rangle_{n}}{n^{2r-\alpha}}=p_{\alpha}(r)\frac{(r-\alpha)^{(r)}}{r!}\left(\frac{\Gamma(\alpha+\theta-r+1)}{\Gamma(\theta+1)}\right)^{2}S_{\alpha,\theta}\qquad\text{a.s.}
\end{equation}
\end{lem}

\begin{proof}
It follows from \eqref{DEFK}, \eqref{DEFKr}, \eqref{DEFXIKr} and \eqref{DEFSNr} that
\begin{align}
\label{eq:esr0}
\E[S_{r,n+1}\,|\,\mathcal{F}_{r,n}]
&=\E\left[\sum_{i=1}^{r}b_{r,i}K_{i,n+1}+(-1)^r p_{\alpha}(r) K_{n+1}\,\Bigr|\,\mathcal{F}_{r,n}\right] \qquad \text{a.s.}
\notag\\
&=\E\left[\sum_{i=1}^{r}b_{r,i}(K_{i,n}+\xi_{i,n+1})+(-1)^r p_{\alpha}(r)(K_{n}+\xi_{n+1})\,\Bigr|\,\mathcal{F}_{r,n}\right]  \qquad \text{a.s.} \notag \\
&=S_{r,n}+ \sum_{i=1}^{r}b_{r,i}(p_{i,n}-q_{i,n}) + (-1)^rp_{\alpha}(r) p_{n}  \qquad \text{a.s.}
\end{align}
Consequently, as $p_{1,n}=p_n$ and for all $i \geq 2$, $p_{i,n}=q_{i-1,n}$, we obtain from \eqref{eq:esr0} that
\begin{equation}
\label{eq:esr1}
\E[S_{r,n+1}\,|\,\mathcal{F}_{r,n}]=S_{r,n}+((-1)^rp_{\alpha}(r)+b_{r,1})p_{n}-q_{r,n}+\sum_{i=1}^{r-1}(b_{r,i+1}-b_{r,i})q_{i,n}  \qquad \text{a.s.}
\end{equation}
Moreover, we clearly have from \eqref{DEFbr} that
\begin{equation*}
b_{r,i+1}= \left(\frac{r-i}{\alpha -i}\right)b_{r,i} \hspace{1cm} \text{and} \hspace{1cm} b_{r,i+1}-b_{r,i}= \left(\frac{r-\alpha}{\alpha -i}\right)b_{r,i}
\end{equation*}
Therefore, we deduce from \eqref{DEFPN}, \eqref{DEFIPr}, \eqref{DEFSNr}, \eqref{eq:esr1} together with straighforward calculation that
\begin{equation}
\label{eq:esr2}
\E[S_{r,n+1}\,|\,\mathcal{F}_{r,n}]=\gamma_{r,n}S_{r,n}-(1-\gamma_{r,n})\frac{(-1)^{r}p_{\alpha}(r) \theta}{\alpha}  \qquad \text{a.s.}
\end{equation}
where
\begin{displaymath}
\gamma_{r,n}=1-\left(\frac{r-\alpha}{n+\theta}\right).
\end{displaymath}
Hereafter, we obtain from \eqref{DEFMr} and \eqref{eq:esr2} and the elementary identity $a_{r,n+1}\gamma_{r,n}= a_{r,n}$
that for all $n\geq1$,
\begin{align*}
\E[M_{r,n+1}\,|\,\mathcal{F}_{r,n}]&=\E\left[a_{r,n+1}\left(S_{r,n+1}+\frac{(-1)^{r} p_{\alpha}(r) \theta}{\alpha}\right)\,\Bigr|\,\mathcal{F}_{r,n}\right] \qquad \text{a.s.}\\
&=a_{r,n+1}\left(\E[S_{r,n+1}\,|\,\mathcal{F}_{r,n}]+\frac{(-1)^{r} p_{\alpha}(r) \theta}{\alpha}\right) \qquad \text{a.s.}\\
&=a_{r,n+1}\left(\gamma_{r,n}S_{r,n}-(1-\gamma_{r,n})\frac{(-1)^{r}p_{\alpha}(r) \theta}{\alpha}+\frac{(-1)^{r} p_{\alpha}(r) \theta}{\alpha}\right) \qquad \text{a.s.}\\
&=a_{r,n+1}\gamma_{r,n} \left(S_{r,n}+\frac{(-1)^{r} p_{\alpha}(r) \theta}{\alpha}\right)  \qquad \text{a.s.}\\
&=a_{r,n} \left(S_{r,n}+\frac{(-1)^{r} p_{\alpha}(r) \theta}{\alpha}\right)= M_{r,n} \qquad \text{a.s.}
\end{align*}
Moreover, the sequence $(M_{r,n})$ is square integrable as for all $1 \leq i \leq r$, $K_{i,n}\leq K_n \leq n$. Consequently, $(M_{r,n})$ is a locally square
integrable martingale. Its predictable quadratic variation $\langle M_r \rangle_{n}$ is given by
\begin{equation}
\label{eq:CVMr1}
\langle M_r \rangle_{n} = \sum_{k=1}^{n-1} \E[(\Delta  M_{r,k+1})^{2}\,|\,\mathcal{F}_{r,k}]
\end{equation}
where $\Delta M_{r,n+1}=M_{r,n+1}-M_{r,n}$. It follows from \eqref{DEFAr}, \eqref{DEFMr} and \eqref{eq:esr2} that
\begin{equation}
\label{DECOMG}
\Delta M_{r,n+1}= a_{r,n+1} \Big( S_{r,n+1} - \E[S_{r,n+1}\,|\,\mathcal{F}_{r,n}]\Big)=a_{r,n+1} \Big( \zeta_{r,n+1} - \E[\zeta_{r,n+1}\,|\,\mathcal{F}_{r,n}]\Big)
\end{equation}
where
\begin{equation*}
\zeta_{r,n+1}= \sum_{i=1}^{r}b_{r,i}\xi_{i,n+1}+(-1)^r p_{\alpha}(r) \xi_{n+1}.
\end{equation*}
Consequently, we immediately obtain from decomposition \eqref{DECOMG} that
\begin{equation}
\label{eq:CVMr2}
\E[(\Delta M_{r,n+1})^{2}\,|\,\mathcal{F}_{r,n}]=a_{r,n+1}^{2}\Big(\E[\zeta_{r,n+1}^{2}\,|\,\mathcal{F}_{r,n}]-\E^2[\zeta_{r,n+1}\,|\,\mathcal{F}_{r,n}]\Big) \qquad \text{a.s.}
\end{equation}
We already saw from \eqref{eq:esr0} that
\begin{equation}
\label{CEzeta1}
\E[\zeta_{r,n+1}\,|\,\mathcal{F}_{r,n}]= \sum_{i=1}^{r}b_{r,i}(p_{i,n}-q_{i,n}) + (-1)^rp_{\alpha}(r) p_{n} \qquad \text{a.s.}
\end{equation}
In addition, we also have $\E[\xi_{n+1}^2\,|\,\mathcal{F}_{r,n}]=p_n$ and $\E[\xi_{i,n+1}^2\,|\,\mathcal{F}_{r,n}]=p_{i,n}+q_{i,n}$ for $1 \leq i \leq r$.
Moreover, $\E[\xi_{i,n+1} \xi_{j,n+1}\,|\,\mathcal{F}_{r,n}]=0$ as soon as $|i-j| \geq 2$. Furthermore, one can observe that $\E[\xi_{i,n+1} \xi_{i+1,n+1}\,|\,\mathcal{F}_{r,n}]=-q_{i, n}$ for $1 \leq i \leq r-1$ and
$\E[\xi_{i,n+1} \xi_{i-1,n+1}\,|\,\mathcal{F}_{r,n}]=-q_{i-1,n}$ for $2 \leq i \leq r$ and
$\E[\xi_{n+1} \xi_{i,n+1} \,|\,\mathcal{F}_{r,n}]=p_n$ for $i=1$, $\E[\xi_{n+1} \xi_{i,n+1} \,|\,\mathcal{F}_{r,n}]=0$ for $2 \leq i \leq r$. Hence,
we have 
\begin{equation}
\label{CEzeta2}
\E[\zeta_{r,n+1}^2\,|\,\mathcal{F}_{r,n}]= \sum_{i=1}^{r}b_{r,i}^2(p_{i,n}+q_{i,n}) -2\sum_{i=1}^{r-1}b_{r,i}b_{r,i+1}q_{i,n} - \Bigl(\frac{2r-\alpha}{\alpha}\Bigr) p_{\alpha}^2(r)p_n
\quad \text{a.s.}
\end{equation}
We have from the almost sure convergences \eqref{eq:kalpha} and \eqref{eq:kralpha} together with \eqref{DEFPN} and \eqref{DEFIPr} that
\begin{equation}
\label{ASCVGPN}
\lim_{n \rightarrow \infty} n^{1-\alpha} p_n= \alpha S_{\alpha, \theta} \hspace{1cm} \text{a.s.}
\end{equation}
and for all $i \geq 1$,
\begin{equation}
\label{ASCVGQIN}
\lim_{n \rightarrow \infty} n^{1-\alpha} q_{i,n}= (i-\alpha) p_{\alpha}(i) S_{\alpha, \theta} \hspace{1cm} \text{a.s.}
\end{equation}
However, for all $i \geq 2$, $p_{i,n}=q_{i-1,n}$. Hence, we obtain from \eqref{ASCVGQIN} that for all $i \geq 2$, 
\begin{align}
\label{ASCVGPQINDif}
\lim_{n \rightarrow \infty} n^{1-\alpha} (p_{i,n}- q_{i,n})
&= ((i-1-\alpha) p_{\alpha}(i-1) -(i-\alpha) p_{\alpha}(i)) S_{\alpha, \theta} \notag
\hspace{1cm} \text{a.s.} \\
&=\alpha p_{\alpha}(i) S_{\alpha, \theta} \hspace{1cm} \text{a.s.}
\end{align}
Therefore, it follows from \eqref{CEzeta1}, \eqref{ASCVGPN} and \eqref{ASCVGPQINDif} that
\begin{equation}
\label{ASCVGPCEzetar1}
\lim_{n \rightarrow \infty} n^{1-\alpha} \E[\zeta_{r,n+1}\,|\,\mathcal{F}_{r,n}]= \alpha  \Big( \sum_{i=1}^{r}b_{r,i} p_{\alpha}(i) + (-1)^r  p_{\alpha}(r) \Big)S_{\alpha, \theta} =0 \hspace{1cm} \text{a.s.}
\end{equation}
since by the Binomial theorem, the sum in \eqref{ASCVGPCEzetar1} reduces to
\begin{align}
\sum_{i=1}^{r}b_{r,i} p_{\alpha}(i) &= \sum_{i=1}^{r}
(-1)^{r-i}\frac{(i-\alpha)^{(r-i)}}{(r-i)!} \frac{\alpha(1-\alpha)^{(i-1)}}{i!} \notag \\
&=\frac{\alpha(1-\alpha)^{(r-1)}}{r!} \sum_{i=1}^{r} \binom{r}{i} (-1)^{r-i} \notag \\
&=-(-1)^r  p_{\alpha}(r).
\label{BinomialBP}
\end{align}
By the same token, we deduce from \eqref{CEzeta2}, \eqref{ASCVGPN} and \eqref{ASCVGQIN} that 
\begin{align}
\label{ASCVGPCEzetar2}
\lim_{n \rightarrow \infty} n^{1-\alpha} \E[\zeta_{r,n+1}^2\,|\,\mathcal{F}_{r,n}]
&= \Big( \sum_{i=1}^{r}b_{r,i}^2 ((2i - \alpha)+2(r-i)) p_{\alpha}(i) - (2r-\alpha)p_{\alpha}^2(r) \Big)S_{\alpha, \theta} \notag \\  
&= (2r- \alpha) \Big( \sum_{i=1}^{r}b_{r,i}^2  p_{\alpha}(i) - p_{\alpha}^2(r) \Big)S_{\alpha, \theta} 
\qquad \text{a.s.} \notag \\  
&=  p_{\alpha}(r) 
\frac{(r-\alpha)^{(r+1)}}{r!} S_{\alpha, \theta}  \qquad \text{a.s.}
\end{align}
since the Chu-Vandermonde identity for the rising factorial
\begin{equation}
\label{ChuVan}
\sum_{k=0}^n (-n)^{(k)} \frac{(b)^{(k)}}{(a)^{(k)} k!}
= \sum_{k=0}^n (-1)^{k} \binom{n}{k} \frac{(b)^{(k)}}{(a)^{(k)}}= \frac{(a-b)^{(n)}}{(a)^{(n)}}
\end{equation}
which holds for all $a,b \in \dR$ with $a \notin \{0,-1,-2,\ldots\}$, implies via \eqref{DEFbr} that
\begin{align*}
\sum_{i=1}^{r}b_{r,i}^2  p_{\alpha}(i)&=\sum_{i=1}^{r}\left(\frac{(i-\alpha)^{(r-i)}}{(r-i)!}\right)^{2}\frac{\alpha(1-\alpha)^{(i-1)}}{i!}\\
&=\Gamma^2(r-\alpha)\sum_{i=1}^{r}\frac{\alpha(1-\alpha)^{(i-1)}}{\Gamma^2(i-\alpha)((r-i)!)^{2}i!}\\
&=-\frac{\Gamma^2(r-\alpha)}{\Gamma^2(-\alpha)(r!)^{2}}\sum_{i=1}^{r}\frac{(-r)^{(i)}(-r)^{(i)}}{(-\alpha)^{(i)}i!}\\
&=-p_\alpha^2(r)\left(\frac{(r-\alpha)^{(r)}}{(-\alpha)^{(r)}}-1\right).
\end{align*}
Hereafter, we obtain from standard results on the asymptotic behavior of
the Euler Gamma function together with \eqref{DEFAr} that
\begin{equation}
\label{CVGANr}
\lim_{n\rightarrow+\infty} \frac{a_{r,n}}{n^{r-\alpha}}=\lim_{n\rightarrow+\infty}\frac{1}{n^{r-\alpha}}
\frac{(\theta+1)^{(n-1)}}{(\alpha+\theta-r+1)^{(n-1)}}=\frac{\Gamma(\alpha+\theta-r+1)}{\Gamma(\theta+1)}.
\end{equation}
Therefore, it follows from \eqref{eq:CVMr2}, \eqref{ASCVGPCEzetar1}, \eqref{ASCVGPCEzetar2} and \eqref{CVGANr} that
\begin{equation*}
\lim_{n\rightarrow+\infty} \frac{\E[(\Delta M_{r,n+1})^{2}\,|\,(\mathcal{F}_{r,n})]}{n^{2r-\alpha -1}}=
p_{\alpha}(r) \frac{(r-\alpha)^{(r+1)}}{r!} \left(\frac{\Gamma(\alpha+\theta-r+1)}{\Gamma(\theta+1)}\right)^2 S_{\alpha, \theta}  \qquad \text{a.s.}
\end{equation*}
which implies through a direct application of Toeplitz lemma that
\begin{equation*}
\lim_{n\rightarrow+\infty}\frac{\langle M_{r}\rangle_{n}}{n^{2r-\alpha}}=p_{\alpha}(r)\frac{(r-\alpha)^{(r)}}{r!}\left(\frac{\Gamma(\alpha+\theta-r+1)}{\Gamma(\theta+1)}\right)^{2}S_{\alpha,\theta}\qquad\text{a.s.}
\end{equation*}
completing the proof of Lemma \ref{L-KEY}.
\end{proof}


\subsection{A central limit theorem} 
\label{Sub-SCLT}


\begin{lem}
\label{L-ScalarCLT}
Assume that $\alpha\in]0,1[$ and  $\alpha+\theta>0$. Then, we have the asymptotic normality
\begin{equation}
\label{LCLT1}
\sqrt{K_n}( P_{1,n} -\alpha )
\underset{n\rightarrow+\infty}{\overset{\cL}{\longrightarrow}}
\cN(0, \alpha(1-\alpha)).
\end{equation} 
Moreover, for all $r\geq 2$, we also have
\begin{equation}
\label{LCLTr}
\sqrt{K_{n}} \left(P_{r,n}   + \sum_{i=1}^{r-1}b_{r,i}P_{i,n} +(-1)^r p_{\alpha}(r) \right)\underset{n\rightarrow+\infty}{\overset{\cL}{\longrightarrow}} \cN\left(0, p_{\alpha}(r)\frac{(r-\alpha)^{(r)}}{r!}\right).
\end{equation}
\end{lem}

\begin{proof}
We just saw in lemma \ref{L-KEY} that the predictable quadratic variation $\langle M_{r} \rangle_n$ satisfies
\begin{equation}
\label{ASCVGPQVr}
\lim_{n\rightarrow+\infty}\frac{\langle M_{r}\rangle_{n}}{n^{2r-\alpha}}=p_{\alpha}(r)\frac{(r-\alpha)^{(r)}}{r!}\left(\frac{\Gamma(\alpha+\theta-r+1)}{\Gamma(\theta+1)}\right)^{2}S_{\alpha,\theta}\qquad\text{a.s.}
\end{equation}
Moreover, we are going to prove that Lindeberg’s condition is satisfied, that is for all $\varepsilon >0$,
\begin{equation}
\label{ASCVGPQVr4}
\frac{1}{n^{2r-\alpha}}\sum_{k=1}^{n}\E\big[\Delta M_{r,k}^2 \rI_{\{|\Delta M_{r, k}|>\varepsilon \sqrt{n^{2r-\alpha}} \}} \,|\,\mathcal{F}_{r,k-1} \big] \overset{\displaystyle \dP}{\underset{n\to\infty}{\longrightarrow}} 0.
\end{equation}
We immediately have for all $\varepsilon >0$,
$$
\sum_{k=1}^{n}\E\big[\Delta M_{r,k}^2 \rI_{\{|\Delta M_{r, k}|>\varepsilon \sqrt{n^{2r-\alpha}} \}} \,|\,\mathcal{F}_{r,k-1}\big]
\leq 
\frac{1}{\varepsilon^2 n^{2r-\alpha}}\sum_{k=1}^{n}\E\big[\Delta M_{r, k}^4 \,|\,\mathcal{F}_{r,k-1}\big].
$$
Hence, in order to prove \eqref{ASCVGPQVr4}, 
it is only necessary to show that
\begin{equation}
\label{ASCVGPQVr5}
 \frac{1}{n^{4r-2\alpha}}\sum_{k=1}^{n}
\E\big[\Delta M_{r, k}^4 \,|\,\mathcal{F}_{r,k-1} \big] \overset{\displaystyle \dP}{\underset{n\to\infty}{\longrightarrow}} 0.
\end{equation}
Using the same approach as in \eqref{CEzeta2}, we obtain that for all $n \geq 1$,
\begin{eqnarray*}
\E[\zeta_{r,n+1}^4\,|\,\mathcal{F}_{r,n}] &=& \sum_{i=1}^{r}b_{r,i}^4(p_{i,n}+q_{i,n}) -4\sum_{i=1}^{r-1}b_{r,i}^3b_{r,i+1}q_{i,n} +6\sum_{i=1}^{r-1}b_{r,i}^{2}b^{2}_{r,i+1}q_{i,n} \\
&-4&
\sum_{i=1}^{r-1}b_{r,i}b_{r,i+1}^3q_{i,n} +
p_{n}\sum_{j=1}^{4}\binom{4}{j}(-1)^{jr}p_{\alpha}^{j}(r)b_{r,1}^{4-j}
\qquad \text{a.s.} 
\end{eqnarray*}
However, we already saw that $b_{r,r}=1$ and for all $1\leq i \leq r-2$, $b_{r,i+1}= -e_{r,i} b_{r,i}$ where
\begin{equation*}
 e_{r,i}= \left(\frac{r-i}{i-\alpha}\right).
\end{equation*}
Consequently,
\begin{eqnarray*}
\E[\zeta_{r,n+1}^4\,|\,\mathcal{F}_{r,n}] &=& \sum_{i=1}^{r}b_{r,i}^4 p_{i,n}+ \sum_{i=1}^{r}b_{r,i}^4 q_{i,n} \Big(1 + 4e_{r,i} + 6e_{r,i}^2 + 4e_{r,i}^3\Big)\\
&+& p_{n}\sum_{j=1}^{4}\binom{4}{j}(-1)^{jr}p_{\alpha}^{j}(r)b_{r,1}^{4-j} 
\qquad \text{a.s.} 
\end{eqnarray*}
which leads to
\begin{equation}
\label{CEzeta4}
\E[\zeta_{r,n+1}^4\,|\,\mathcal{F}_{r,n}] \leq \sum_{i=1}^{r}b_{r,i}^4 p_{i,n} + \left(\frac{r-\alpha}{1-\alpha}\right)^{\!\!4} \sum_{i=1}^{r}b_{r,i}^4 q_{i,n}
+ p_n p_{\alpha}^4(r) \left(\frac{r-\alpha}{\alpha}\right)^{\!\!4}  
\quad \text{a.s.} 
\end{equation}
Therefore, we deduce from \eqref{DECOMG} and \eqref{CEzeta4} together with
the almost sure convergences \eqref{ASCVGPN}, \eqref{ASCVGQIN} and \eqref{CVGANr} that
\begin{equation}
\label{CEMG4}
\sum_{k=1}^{n}\E\big[\Delta M_{r, k}^4 \,|\,\mathcal{F}_{r,k-1}\big]=O\left( \sum_{k=1}^n k^{4r-1-3\alpha} \right)=O(n^{4r-3\alpha}) \qquad\text{a.s.}
\end{equation}
which clearly implies \eqref{ASCVGPQVr5}.
Hence, the two conditions of \citep[Corollary 3.1]{Hal(80)}
are satisfied, which allows us to conclude that for all $r \geq 1$,
\begin{equation}
\label{ASCVGPQVr6}
\frac{M_{r,n}}{\sqrt{\langle M_{r} \rangle_n}} \underset{n\rightarrow+\infty}{\overset{\cL}{\longrightarrow}} \cN(0, 1),
\end{equation}
Furthermore, we find from \eqref{DEFMr}, \eqref{DEFSNr} and \eqref{ASCVGPQVr6} that for all $r \geq 1$,
\begin{equation*}
\frac{a_{r,n}}{\sqrt{\langle M_{r} \rangle_n}} \left(S_{r,n}+\frac{(-1)^{r} p_{\alpha}(r) \theta}{\alpha}\right)\underset{n\rightarrow+\infty}{\overset{\cL}{\longrightarrow}} \cN(0, 1),
\end{equation*}
\begin{equation*}
\frac{a_{r,n}}{\sqrt{\langle M_{r} \rangle_n}} \left( \sum_{i=1}^{r}b_{r,i}K_{i,n} +(-1)^r p_{\alpha}(r) K_{n}  +\frac{(-1)^{r} p_{\alpha}(r) \theta}{\alpha}\right)\underset{n\rightarrow+\infty}{\overset{\cL}{\longrightarrow}} \cN(0, 1),
\end{equation*}
\begin{equation}
\label{CLTPrPR1}
\frac{a_{r,n} K_n}{\sqrt{\langle M_{r} \rangle_n}} \left( \sum_{i=1}^{r}b_{r,i}P_{i,n} +(-1)^r p_{\alpha}(r)  +\frac{(-1)^{r} p_{\alpha}(r) \theta}{\alpha K_n}\right)\underset{n\rightarrow+\infty}{\overset{\cL}{\longrightarrow}} \cN(0, 1).
\end{equation}
However, we clearly have from \eqref{eq:kalpha}, \eqref{CVGANr} and \eqref{ASCVGPQVr} that
\begin{equation}
\label{CLTPrPR2}
\lim_{n\rightarrow+\infty}\frac{\langle M_{r}\rangle_{n}}{a^{2}_{r,n}K_{n}}=p_{\alpha}(r)\frac{(r-\alpha)^{(r)}}{r!}\qquad\text{a.s.}
\end{equation}
and
\begin{equation}
\label{CLTPrPR3}
\lim_{n\rightarrow+\infty}\frac{(-1)^{r} p_{\alpha}(r) \theta}{\alpha K_n}=0 \qquad\text{a.s.}
\end{equation}
Finally, we find from \eqref{CLTPrPR1}, \eqref{CLTPrPR2}, \eqref{CLTPrPR3} and Slutsky’s lemma that
\begin{equation}
\label{CLTPrPR4}
\sqrt{K_{n}} \left( \sum_{i=1}^{r}b_{r,i}P_{i,n} +(-1)^r p_{\alpha}(r) \right)\underset{n\rightarrow+\infty}{\overset{\cL}{\longrightarrow}} \cN\left(0, p_{\alpha}(r)\frac{(r-\alpha)^{(r)}}{r!}\right),
\end{equation}
which achieves the proof of Lemma \ref{L-ScalarCLT}.
\end{proof}


\subsection{The multi-dimensional martingale construction}
\label{Sub-MDMC}


Consider the $d$-dimensional random vector $\cM_{n}$ defined as 
\begin{equation}
\label{DEFVMr}
\cM_{n}= \begin{pmatrix}
M_{1,n} \\
M_{2,n} \\
\vdots \\
M_{d,n}
\end{pmatrix}
\end{equation}
where each component $M_{r,n}$ has been previously defined in \eqref{DEFMr}.
We already in Lemma \ref{L-KEY} that for all $r\geq 1$, $(M_{r,n})$ is a locally square integrable martingale.
It implies that the sequence $(\cM_{n})$ is a locally square integrable $d$-dimensional martingale. 
Let $(V_n)$ be the sequence of deterministic square matrices of order $d$ given by

\begin{equation}
\label{DEFmatrixVN}
V_{n}=n^{\alpha/2} \text{diag}\left(\frac{1}{n}, \frac{1}{n^2}, \ldots, \frac{1}{n^d}\right).
\end{equation}

\begin{lem}
\label{L-KEYMATRIX}
The predictable quadratic variation $\langle \cM\rangle_{n}$ of the martingale $(\cM_{n})$ satisfies
\begin{equation}
\label{LIMmatrixMd}
\lim_{n\rightarrow+\infty}V_{n}\langle \cM\rangle_{n}V_{n}= S_{\alpha,\theta} \Sigma_d \qquad\text{a.s.}
\end{equation}
where $\Sigma_d$ stands for the square matrix of order $d$ whose $(i,j)$-th element is given by
\begin{equation}
\label{Sigmaij}
\Sigma_{i,j}=\frac{(-1)^{i+j} (i+j)!}{i!j!}p_\alpha(i+j)\frac{\Gamma(\alpha+\theta-i+1) \Gamma(\alpha+\theta-j+1)}{\Gamma^2(\theta+1)}.
\end{equation}
\end{lem}

\begin{proof}
We shall proceed as in the proof of Lemma \ref{L-KEY}. We deduce from the initial decomposition \eqref{DECOMG} that 
\begin{equation}
\label{DEFmatrixIPMd}
\langle \cM\rangle_{n}=\sum_{k=1}^{n-1}\E[\Delta \cM_{k+1}\Delta \cM_{k+1}^{T}\,|\,\mathcal{F}_{d,k}]
\end{equation}
where
\begin{displaymath}
\Delta \cM_{n+1}=
\begin{pmatrix}
\Delta M_{1,n+1}\\
\Delta M_{2,n+1}\\
\vdots\\
\Delta M_{d,n+1}
\end{pmatrix}
=
\begin{pmatrix}
a_{1,n+1}(S_{1,n+1}-\E[S_{1,n+1}\,|\,\mathcal{F}_{d,n}])\\
a_{2,n+1}(S_{2,n+1}-\E[S_{2,n+1}\,|\,\mathcal{F}_{d,n}])\\
\vdots\\
a_{d,n+1}(S_{d,n+1}-\E[S_{d,n+1}\,|\,\mathcal{F}_{d,n}])\\
\end{pmatrix}.
\end{displaymath}
By the same calculation as in \eqref{eq:CVMr2}, we have for all $1\leq i,j \leq d$,
\begin{align}\label{eq:start}
& \E[\Delta M_{i,n+1}\Delta M_{j,n+1}\,|\,\mathcal{F}_{d,n}]\\
&\notag\quad=a_{i,n+1}a_{j,n+1}\big(\E[S_{i,n+1}S_{j,n+1}\,|\,\mathcal{F}_{d,n}]-\E[S_{i,n+1}\,|\,\mathcal{F}_{d,n}]\E[S_{j,n+1}\,|\,\mathcal{F}_{d,n}]\big) 
\\
&\notag\quad=a_{i,n+1}a_{j,n+1}\big(c_{(i,j),n}-b_{(i,j),n}\big)
\end{align}
where $b_{(i,j),n}=\E[\zeta_{i,n+1} \,|\,\mathcal{F}_{d,n}] \E[\zeta_{j,n+1} \,|\,\mathcal{F}_{d,n}]$ and
$c_{(i,j),n}=\E[\zeta_{i,n+1} \zeta_{j,n+1} \,|\,\mathcal{F}_{d,n}]$.
On the one hand, it follows from \eqref{CEzeta1} that for all $1\leq i,j \leq d$,
\begin{align}\label{equation_partb}
& b_{(i,j),n}= \sum_{t=1}^{i}\sum_{s=1}^{j}b_{i,t}b_{j,s}(p_{t,n}-q_{t,n})(p_{s,n}-q_{s,n})
+(-1)^{i+j}p_{\alpha}(i)p_{\alpha}(j)p_{n}^{2}\\
&\notag+(-1)^{j}p_{\alpha}(j)p_{n}\sum_{t=1}^{i}b_{i,t}(p_{t,n}-q_{t,n})+(-1)^{i}p_{\alpha}(i)p_{n}\sum_{s=1}^{j}b_{j,s}(p_{s,n}-q_{s,n}) \qquad \text{a.s.}
\end{align}
On the other hand, we deduce from the same calculation as in \eqref{CEzeta2} that for all $1\leq i,j \leq d$ such that $i \leq j$,
\begin{align}\label{equation_partc}
 c_{(i,j),n}&=\sum_{t=1}^{i}b_{i,t}b_{j,t} (p_{t,n}+q_{t,n})
- \sum_{t=1}^{i-1}(b_{i,t}b_{j,t+1}+b_{j,t}b_{i,t+1})q_{t,n}  -b_{j,i+1}q_{i,n}\\
&\notag - \Bigl(\frac{i+j-\alpha}{\alpha}\Bigr) (-1)^{i+j} p_{\alpha}(i)p_{\alpha}(j)p_n \qquad \text{a.s.}
\end{align}
Convergence \eqref{ASCVGPCEzetar1} of course ensures that for all $1\leq i,j \leq d$,
\begin{equation}
\label{cvg-partb}
\lim_{n \rightarrow \infty} n^{1-\alpha} b_{(i,j),n} = 0 \hspace{1cm} \text{a.s.}
\end{equation}
Moreover, we obtain from \eqref{ASCVGPN} and \eqref{ASCVGQIN} that for all $1\leq i,j \leq d$ such that $i \leq j$,
\begin{align}\label{cvg-partc}
\lim_{n \rightarrow \infty} n^{1-\alpha} c_{(i,j),n}
&= \Big( \sum_{t=1}^{i}b_{i,t} b_{j,t} (2t - \alpha) p_{\alpha}(t) 
+  \sum_{t=1}^{i-1}b_{i,t} b_{j,t} (i+j -2t) p_{\alpha}(t) \Big)S_{\alpha, \theta} \notag \\
&\notag -b_{j,i+1}(i-\alpha)p_{\alpha}(i)S_{\alpha,\theta} - (i+j-\alpha) (-1)^{i+j} p_{\alpha}(i)p_{\alpha}(j) S_{\alpha, \theta} 
\qquad \text{a.s.} \notag \\  
&\notag=(i+j- \alpha) \Big( \sum_{t=1}^{i}b_{i,t} b_{j,t}  p_{\alpha}(t) - (-1)^{i+j} p_{\alpha}(i)p_{\alpha}(j) \Big)S_{\alpha, \theta} \\ 
&\notag - \Big( b_{j,i+1}(i-\alpha)+b_{j,i}(j-i) \Big) p_\alpha(i) S_{\alpha, \theta}   \qquad \text{a.s.}\notag\\
&=(i+j- \alpha) \Big( \sum_{t=1}^{i}b_{i,t} b_{j,t}  p_{\alpha}(t) - (-1)^{i+j} p_{\alpha}(i)p_{\alpha}(j) \Big)S_{\alpha, \theta}
\end{align}
Furthermore, the sum in \eqref{cvg-partc} reduces to
\begin{align}\label{final0}
\notag\sum_{t=1}^{i}b_{i,t} b_{j,t}  p_{\alpha}(t)&=(-1)^{i+j}\alpha\sum_{t=1}^{i}\frac{(t-\alpha)^{(i-t)}(t-\alpha)^{(j-t)}(1-\alpha)^{(t-1)}}{t!(i-t)!(j-t)!}\\
&\notag=(-1)^{i+j}\frac{\Gamma(i-\alpha)\Gamma(j-\alpha)}{\Gamma(1-\alpha)}\frac{\alpha}{i!j!}\sum_{t=1}^{i}\frac{i!(j-t+1)^{(t)}}{t!(i-t)!\Gamma(t-\alpha)}\\
&=(-1)^{i+j}\frac{\Gamma(i-\alpha)\Gamma(j-\alpha)}{\Gamma(1-\alpha)\Gamma(-\alpha)}\frac{\alpha}{i!j!}\sum_{t=1}^{i}\frac{(-i)^{(t)}(-j)^{(t)}}{(-\alpha)^{(t)}t!}.
\end{align}
By applying one again Chu-Vandermonde identity \eqref{ChuVan} to \eqref{final0}, we obtain that
\begin{align}\label{final1}
\notag\sum_{t=1}^{i}b_{i,t} b_{j,t}  p_{\alpha}(t)&=(-1)^{i+j}\frac{\Gamma(i-\alpha)\Gamma(j-\alpha)}{\Gamma(1-\alpha)\Gamma(-\alpha)}\frac{\alpha}{i!j!}\left(-1+\frac{(j-\alpha)^{(i)}}{(-\alpha)^{(i)}}\right)\\
&\notag\quad=(-1)^{i+j+1}\frac{\Gamma(i-\alpha)\Gamma(j-\alpha)}{\Gamma(-\alpha)\Gamma(-\alpha)}\frac{1}{i!j!}\left(-1+\frac{(j-\alpha)^{(i)}}{(-\alpha)^{(i)}}\right)\\
&\quad=(-1)^{i+j+1}p_{\alpha}(i)p_{\alpha}(j)\left(\frac{(j-\alpha)^{(i)}}{(-\alpha)^{(i)}}-1\right).
\end{align}
Therefore, we deduce from \eqref{cvg-partc} and \eqref{final1} that for all $1\leq i,j \leq d$ such that $i \leq j$,
\begin{eqnarray}
\lim_{n \rightarrow \infty} n^{1-\alpha} c_{(i,j),n}
&=& (-1)^{i+j+1} (i+j- \alpha)\frac{(j-\alpha)^{(i)}}{(-\alpha)^{(i)}}p_{\alpha}(i)p_{\alpha}(j)S_{\alpha, \theta} \notag\\
&=&\frac{(-1)^{i+j} (i+j)!}{i!j!}(i+j- \alpha)p_\alpha(i+j)S_{\alpha,\theta}
\hspace{1cm} \text{a.s.}
\label{final2}
\end{eqnarray}
One can observe that \eqref{final2} coincides with \eqref{ASCVGPCEzetar2} in the special case $i=j=r$.
Finally, it follows from \eqref{CVGANr}, \eqref{eq:start} and \eqref{final2} that for all $1\leq i,j \leq d$ such that $i \leq j$,
\begin{equation}
\lim_{n\rightarrow+\infty} \frac{\E[\Delta M_{i,n+1}\Delta M_{j,n+1}\,|\,\mathcal{F}_{d,n}]}{n^{i+j-\alpha -1}}=
(i+j-\alpha) \Sigma_{i,j} S_{\alpha, \theta}  \hspace{1cm} \text{a.s.}
\label{final3}
\end{equation}
where the covariance term $\Sigma_{i,j}$ has been previously defined in \eqref{Sigmaij}. Convergence \eqref{final3} leads, via Toeplitz lemma, to the conclusion that for all $1\leq i,j \leq d$ such that $i \leq j$,
\begin{equation*}
\lim_{n\rightarrow+\infty}\frac{1}{n^{i+j-\alpha}} \sum_{k=1}^n
\E[\Delta M_{i,k+1}\Delta M_{j,k+1}\,|\,\mathcal{F}_{d,k}]
=\Sigma_{i,j} S_{\alpha, \theta}  \hspace{1cm} \text{a.s.}
\end{equation*}
which completes the proof of Lemma \ref{L-KEYMATRIX}.
\end{proof}


\subsection{A multi-dimensional central limit theorem} 
\label{Sub-MCLT}


\begin{lem}
\label{L-MATRIXCLT}
Assume that $\alpha\in]0,1[$ and  $\alpha+\theta>0$. Then, we have 
the multi-dimensional asymptotic normality
\begin{equation}
\label{MCLT}
\sqrt{K_{n}} 
\begin{pmatrix}
P_{1,n} -p_{\alpha}(1) \vspace{1ex}\\
P_{2,n} +b_{2,1} P_{1,n} + p_{\alpha}(2) \\
\vdots\\
\displaystyle{P_{d,n}  + \sum_{i=1}^{d-1}b_{d,i}P_{i,n} +(-1)^d p_{\alpha}(d)}
\end{pmatrix}
\underset{n\rightarrow+\infty}{\overset{\cL}{\longrightarrow}} \cN_d(0, \Lambda_d)
\end{equation}
where $\Lambda_d$ denotes the square matrix of order $d$ whose $(i,j)$-th element is given by
\begin{equation}
\label{Lambdaij}
\Lambda_{i,j}=\frac{(-1)^{i+j} (i+j)!}{i!j!}p_\alpha(i+j),
\end{equation}
and $ \cN_d(0, \Lambda_d)$ stands for a centered $d$-dimensional Gaussian random variable with covariance $\Lambda_d$.
\end{lem}

\begin{proof}
The proof of Lemma \ref{L-MATRIXCLT} relies on a non-standard central limit theorem for multi-dimensional martingales given by \citep[Theorem 1]{Tou(91)}. We already saw in Lemma \ref{L-KEYMATRIX} that the predictable quadratic variation
$\langle \cM\rangle_{n}$ satisfies
\begin{equation}
\label{PRCLTMATRIX1}
\lim_{n\rightarrow+\infty}V_{n}\langle \cM\rangle_{n}V_{n}= S_{\alpha,\theta} \Sigma_d \qquad\text{a.s.}
\end{equation}
where the sequence $(V_n)$ is defined in \eqref{DEFmatrixVN}. It is clear that $||V_n||$ decreases to $0$ as $n$ goes to infinity. Hence, it only remains to prove that Lindeberg’s condition is satisfied, that is for all $\varepsilon >0$, 
\begin{equation*}
 \sum_{k=1}^n \E\bigl[\|V_n \Delta \cM_k \|^2 \rI_{\{\|V_n\Delta \cM_k \|>\veps\}}\bigl|\cF_{d,k-1}\bigr] \overset{\displaystyle \dP}{\underset{n\to\infty}{\longrightarrow}}   0.
\end{equation*}
As in the proof of Lemma \ref{L-ScalarCLT}, we have for all $\varepsilon >0$,
$$
\sum_{k=1}^n \E\bigl[\|V_n \Delta \cM_k \|^2 \rI_{\{\|V_n\Delta \cM_k \|>\veps\}}\bigl|\cF_{d,k-1}\bigr]
\leq 
\frac{1}{\varepsilon^2}
\sum_{k=1}^n \E\bigl[\|V_n \Delta \cM_k \|^4 \bigl|\cF_{d,k-1}\bigr].
$$
Consequently, it is only necessary to prove that
\begin{equation}
\label{PRCLTMATRIX2}
\sum_{k=1}^{n}
\E\bigl[\|V_n \Delta \cM_k \|^4 \bigl|\cF_{d,k-1}\bigr] \overset{\displaystyle \dP}{\underset{n\to\infty}{\longrightarrow}} 0.
\end{equation}
We have from \eqref{DEFmatrixVN} that for all $1 \leq k \leq n$,
\begin{equation*}
\|V_n \Delta \cM_k \|^4 =n^{2\alpha} \left(\sum_{i=1}^d \frac{\Delta M_{i,k}^2}{n^{2i}}\right)^2
\leq d n^{2\alpha} \sum_{i=1}^d \frac{\Delta M_{i,k}^4}{n^{4i}},
\end{equation*}
which implies that
\begin{equation}
\label{PRCLTMATRIX3}
\sum_{k=1}^{n}
\E\bigl[\|V_n \Delta \cM_k \|^4 \bigl|\cF_{d,k-1}\bigr] \leq
d n^{2\alpha} \sum_{i=1}^d \frac{1}{n^{4i}} \sum_{k=1}^{n} \E\bigl[\Delta M_{i,k}^4 \bigl|\cF_{d,k-1}\bigr].
\end{equation}
Hereafter, we deduce from \eqref{CEMG4} and \eqref{PRCLTMATRIX3} that
\begin{equation*}
\sum_{k=1}^{n}
\E\bigl[\|V_n \Delta \cM_k \|^4 \bigl|\cF_{d,k-1}\bigr] 
=O\left( \frac{1}{n^\alpha} \right) \qquad\text{a.s.}
\end{equation*}
which clearly leads to \eqref{PRCLTMATRIX2}. Therefore, we obtain from \citep[Theorem 1]{Tou(91)}, see also
\citep{Tou(94)}, that
\begin{equation}
\label{PRCLTMATRIX4}
V_n \cM_n \underset{n\rightarrow+\infty}{\overset{\cL}{\longrightarrow}} \sqrt{S_{\alpha, \theta}} \cN_d(0, \Sigma_d).
\end{equation}
Hence, we find from \eqref{DEFMr}, \eqref{CVGANr} and \eqref{DEFmatrixVN} that
\begin{equation}
\label{PRCLTMATRIX5}
n^{-\alpha/2} 
\begin{pmatrix}
S_{1,n} \\
S_{2,n}  \\
\vdots\\
S_{d,n} 
\end{pmatrix}
\underset{n\rightarrow+\infty}{\overset{\cL}{\longrightarrow}} \sqrt{S_{\alpha, \theta}} \cN_d(0, \Lambda_d).
\end{equation}
Finally, the multi-dimensional asymptotic normality \eqref{MCLT} follows from 
the almost sure convergence \eqref{eq:kalpha} together with \eqref{DEFSNr}, \eqref{PRCLTMATRIX5} 
and Slutsky’s lemma, which achieves the proof of Lemma \ref{L-MATRIXCLT}.
\end{proof}


\subsection{A direct application of the delta-method} 
\label{Sub-COV}


From Lemma \ref{L-MATRIXCLT}, a multi-dimensional central limit theorem for the first $d$-components of the self-normalized Ewens-Pitman process is as follows

\begin{lem}
\label{L-MATRIXCLTEP}
Assume that $\alpha\in]0,1[$ and $\alpha+\theta>0$. Then, we have 
the multi-dimensional central limit theorem
\begin{equation}
\label{MCLTEP}
\sqrt{K_{n}} 
\begin{pmatrix}
P_{1,n} -p_{\alpha}(1) \vspace{2ex}\\
P_{2,n} - p_{\alpha}(2) \\
\vdots\\
P_{d,n} - p_{\alpha}(d)
\end{pmatrix}
\underset{n\rightarrow+\infty}{\overset{\cL}{\longrightarrow}} \cN_d(0, \Gamma_d)
\end{equation}
where $\Gamma_d=J_d \Lambda_d J_d^T$ and $J_d$ stands for the square matrix of order $d$ such that
\begin{equation}
\label{Jij}
J_{i,j}=\left \{ \begin{array}{ccc}
    \frac{(j-\alpha)^{(i-j)}}{(i-j)!}  & \text{ if } & \ \ 1 \leq j \leq i \vspace{2ex}\\
     0 & \text{ if }  & \ \ j>i.
   \end{array}  \right.
\end{equation}
and $ \cN_d(0, \Gamma_d)$ stands for a centered $d$-dimensional Gaussian random variable with covariance $\Gamma_d$.
\end{lem}

\begin{proof}
Let $g$ be the function defined from $\dR^d$ to $\dR^d$ by 
\begin{equation}
g(x_{1}, x_{2}, \cdots, x_{d})
=\left(x_{1}, x_{2}-b_{2,1}x_{1}, \cdots, x_{d} + \sum_{i=1}^{d-1}(-1)^{d-i} b_{d,i}x_{i} 
\right).
\end{equation}
One can easily check that
\begin{equation*}
g\left(x_{1}, x_{2} + b_{2,1}x_{1}, \cdots, x_{d} + \sum_{i=1}^{d-1} b_{d,i}x_{i} \right)
=(x_{1}, x_{2},\cdots, x_{d}) 
.
\end{equation*}
Moreover, it follows from \eqref{BinomialBP} that
\begin{equation*}
g\Big(p_{\alpha}(1), -p_{\alpha}(2),
\cdots, (-1)^{d-1}p_{\alpha}(d)\Big)=
(p_{\alpha}(1), p_{\alpha}(2), \cdots, p_{\alpha}(d))
.
\end{equation*}
Since $\nabla g(x_{1}, x_{2},\cdots, x_{d})=J_d^T$, we immediately obtain \eqref{MCLTEP} 
from \eqref{MCLT} through a direct application of the delta-method.  
\end{proof}

\subsection{Determination of the infinite-dimensional covariance matrix} 
\label{Sub-COV}


We already saw in Lemma \ref{L-MATRIXCLTEP} that the asymptotic covariance matrix $\Gamma_d$ is given by
\begin{equation}
\label{DEFGamma}
\Gamma_d=J_d \Lambda_d J_d^T
\end{equation}
where the  $(i,j)$-th elements of the matrices $\Lambda_d$ and $J_d$ are given by 
\begin{equation*}
\Lambda_{i,j}= \frac{(-1)^{i+j}(i+j)!}{i!j!}p_{\alpha}(i+j)
\end{equation*}
and
\begin{equation*}
J_{i,j}=\left \{ \begin{array}{ccc}
    \frac{(j-\alpha)^{(i-j)}}{(i-j)!}  & \text{ if } & \ \ 1 \leq j \leq i \vspace{2ex}\\
     0 & \text{ if }  & \ \ j>i.
   \end{array}  \right.
\end{equation*}
Denote by $\Gamma_\alpha$ the infinite-dimensional covariance matrix whose section of order $d$ is given by $\Gamma_d$.
The expression of $\Gamma_\alpha$ can be drastically simplified as follows.

\begin{lem}
\label{L-COV}
The infinite-dimensional covariance matrix $\Gamma_\alpha$ is given by
\begin{equation}
\label{EXPGammad}
\Gamma_\alpha=diag(p_\alpha)-p_\alpha p_\alpha^T
\end{equation}
where $p_\alpha=(p_\alpha(1),p_\alpha(2),\ldots)$ stands for the vector containing all the components of the Sibuya distribution.
\end{lem}

\begin{proof}
First of all, it follows from Newton’s generalized Binomial theorem that the generating function of the rising factorial satisfies, for all $a\in \dR$ and $z \in \dC$ wsuch that $|z|<1$, 
\begin{equation}
\label{GFRising}
\sum_{n=0}^\infty \frac{(a)^{(n)}}{n!} z^n= \frac{1}{(1-z)^a}.
\end{equation}
Denote by $G$ the probability generating function of the Sibuya distribution. We obtain from \eqref{Sibuya} that for all $s \in \dR$ such that $|s|<1$,
\begin{equation}
\label{PGF-Sibuya}
G(s)= \sum_{n=1}^\infty p_\alpha(n) s^n=\sum_{n=1}^\infty \frac{\alpha (1-\alpha)^{(n-1)}}{n!} s^n
=-\sum_{n=1}^\infty \frac{(-\alpha)^{(n)}}{n!} s^n
=1 -(1-s)^\alpha.
\end{equation}
Let $F$ be the bivariate generating function defined, for all $s,t \in \dR$ with $|s|<1$ and $|t|<1$, by
\begin{equation}
\label{GF-Gamma}
F(s,t)= \sum_{i=1}^\infty \sum_{j=1}^\infty \Gamma_{i,j} s^i t^j.
\end{equation}
We have from \eqref{DEFGamma} that
\begin{align}
F(s,t) &= \sum_{i=1}^\infty \sum_{j=1}^\infty \sum_{k=1}^i \sum_{\ell=1}^j (-1)^{k+\ell} \binom{k+\ell}{k} 
p_\alpha(k+\ell) \frac{(k-\alpha)^{(i-k)}}{(i-k)!} \frac{(\ell-\alpha)^{(j-\ell)}}{(j-\ell)!} s^i t^j \notag \\
&= \sum_{k=1}^\infty \sum_{\ell=1}^\infty (-1)^{k+\ell} \binom{k+\ell}{k} p_\alpha(k+\ell) \sum_{i=k}^\infty \frac{(k-\alpha)^{(i-k)}}{(i-k)!} s^i
\sum_{j=\ell}^\infty \frac{(\ell-\alpha)^{(j-\ell)}}{(j-\ell)!} t^j.
\label{eq:GFF1}
\end{align}
However, identity \eqref{GFRising} leads to
\begin{equation*}
\sum_{i=k}^\infty \frac{(k-\alpha)^{(i-k)}}{(i-k)!} s^i=s^k \sum_{i=0}^\infty \frac{(k-\alpha)^{(i)}}{i!} s^i= \frac{s^k}{(1-s)^{k-\alpha}}
\end{equation*}
as well as to
\begin{equation*}
\sum_{j=\ell}^\infty \frac{(\ell-\alpha)^{(j-\ell)}}{(j-\ell)!} t^j=t^\ell \sum_{j=0}^\infty \frac{(\ell-\alpha)^{(j)}}{j!} t^j= \frac{t^\ell}{(1-t)^{\ell-\alpha}}.
\end{equation*}
Consequently, by choosing
$$
a_s=\frac{s}{1-s} \hspace{2cm} \text{and} \hspace{2cm} b_t=\frac{t}{1-t}
$$
we deduce from \eqref{eq:GFF1} that for all $s,t \in \dR$ such that $|s|<1$ and $|t|<1$,
\begin{align}
F(s,t) &= (1-s)^\alpha (1-t)^\alpha \sum_{k=1}^\infty \sum_{\ell=1}^\infty (-1)^{k+\ell} \binom{k+\ell}{k} 
p_\alpha(k+\ell) a_s^k b_t^\ell \notag \\
&= (1-s)^\alpha (1-t)^\alpha \sum_{n=2}^\infty \sum_{k=1}^{n-1} (-1)^{n} \binom{n}{k} p_\alpha(n) a_s^k b_t^{n-k} \notag \\
&= (1-s)^\alpha (1-t)^\alpha \sum_{n=2}^\infty  (-1)^{n} p_\alpha(n) \sum_{k=1}^{n-1} \binom{n}{k}  a_s^k b_t^{n-k} \notag \\
&= \alpha (1-s)^\alpha (1-t)^\alpha \sum_{n=2}^\infty  \frac{(-1)^{n} (1-\alpha)^{(n-1)}}{n!} \Big( (a_s+b_t)^n - a_s^n - b_t^n \Big).  
\label{eq:GFF2}
\end{align}
Furthermore, by integrating \eqref{GFRising}, we obtain that for all $a\in \dR$ with $a \neq 1$ and for all $z \in \dC$ such that $|z|<1$, 
\begin{equation}
\label{IGFRising}
\sum_{n=1}^\infty \frac{(a)^{(n)}}{(n+1)!} z^{n+1}= -\frac{1}{1-a} \Big( (1-z)^{1-a} -1 \Big),
\end{equation}
which implies that
\begin{equation}
\label{IGFRisingdev}
\sum_{n=2}^\infty \frac{(-1)^n (a)^{(n-1)}}{n!} z^{n}= z -\frac{1}{1-a} \Big( (1+z)^{1-a} -1 \Big).
\end{equation}
Therefore, it follows from \eqref{PGF-Sibuya}, \eqref{eq:GFF2} and \eqref{IGFRisingdev} that for all $s,t \in \dR$ such that $|s|<1$ and $|t|<1$,
\begin{align}
F(s,t)&= (1-s)^\alpha (1-t)^\alpha \Big( (1+a_s)^{\alpha} + (1+b_t)^{\alpha} - (1+a_s+b_t)^{\alpha} -1  \Big) \notag \\
&= (1-s)^\alpha + (1-t)^\alpha -(1-st)^\alpha - (1-s)^\alpha (1-t)^\alpha \notag \\
&= 1-(1-st)^\alpha - \Big( 1-(1-s)^\alpha \Big) \Big( 1-(1-t)^\alpha \Big) \notag \\
&= G(st)-G(s)G(t)
\label{eq:GFF3}
\end{align}
since
$$
1+a_s= \frac{1}{1-s}, \hspace{1cm} 1+b_t= \frac{1}{1-t}, \hspace{1cm} 1+a_s+b_t= \frac{1-st}{(1-s)(1-t)}.
$$
Finally, we deduce from \eqref{PGF-Sibuya}, \eqref{GF-Gamma} and \eqref{eq:GFF3} that for all $s,t \in \dR$ such that $|s|<1$ and $|t|<1$,
\begin{equation*}
F(s,t) = \sum_{i=1}^\infty \sum_{j=1}^\infty \Gamma_{i,j} s^i t^j= \sum_{i=1}^\infty p_\alpha(i) (st)^i - 
\sum_{i=1}^\infty \sum_{j=1}^\infty p_\alpha(i) p_\alpha(j) s^i t^j.
\end{equation*}
It immediately leads to 
\begin{equation*}
\Gamma_{i,j}=\left \{ \begin{array}{ccc}
    p_\alpha(i)(1-p_\alpha(i)) & \text{ if } & \ \ i=j \vspace{2ex}\\
     -p_\alpha(i)p_\alpha(j) & \text{ if }  & \ \ i \neq j
   \end{array}  \right.
\end{equation*}
which completes the proof of Lemma \ref{L-COV}.
\end{proof}

\subsection{An infinite-dimensional central limit theorem}
\label{Sub-L2}

For all $n\geq1$, denote by $Q_n$ the random element in $\ell^{2}$ defined by $Q_n=(Q_{1,n},\,Q_{2,n},\ldots)$  where, for
all $r \geq 1$, 
\begin{displaymath}
Q_{r,n}=\sqrt{K_{n}}\big(P_{r,n}-p_{\alpha}(r)\big).
\end{displaymath}
For all $d\geq 1$, let $\Pi_d$ be the orthogonal projection from $\dR^\infty$ onto $\dR^d$. Thanks to
Lemma \ref{L-MATRIXCLTEP}, we already  know that for all $d\geq 1$, $\Pi_d(Q_n)$ converges in distribution to 
a centered Gaussian random vector with covariance matrix $\Gamma_d$, that is the weak convergence of the finite-dimensional distributions of $Q_n$. Lemma \ref{L-COV} shows that the infinite-dimensional covariance matrix
$\Gamma_\alpha$ is trace-class as $\text{Tr}(\Gamma_\alpha) \leq 1$. Moreover, it follows from the Kolmogorov extension theorem that there exists a unique centered Gaussian process $Q_\alpha=\cG(\Gamma_\alpha)$ with covariance matrix $\Gamma_\alpha$.

In order to complete the proof of Theorem \ref{main_teo}, it only remains to show that the sequence 
$(Q_n)$ is tight, see e.g. \citet[Chapter 10]{Led(11)} and the references therein. It will be done thanks to our latest lemma.

\begin{lem}
\label{L-IDCLT}
The infinite-dimensional sequence $(Q_n)$ of $\ell^{2}$-valued random variables is tight. More precisely, we have
\begin{equation}\label{toprove}
\lim_{d\rightarrow+\infty}\sup_{n}\sum_{r=d}^\infty\E[(Q_{r,n})^{2}]=0.
\end{equation}
\end{lem}

\begin{proof}
The proof of Lemma \ref{L-IDCLT} relies on \citet[Theorem 2.3.3 and Corollary 2.3.1]{Kal(95)}.
We have from \eqref{DEFPrn} that
\begin{equation}
\label{NORML2}
\| Q_n \|_{\ell^{2}}^2= \sum_{r=1}^\infty Q_{r,n}^2= \sum_{r=1}^\infty K_{n}\big(P_{r,n}-p_{\alpha}(r)\big)^2
=\sum_{r=1}^\infty \frac{L_{r,n}^2}{K_{n}}
\end{equation}
where $L_{r,n}=K_{r,n}-p_{\alpha}(r)K_n$.
Hence, for all $r \geq 1$,
\begin{equation}\label{eq:main}
\E[Q_{r,n}^{2}]=\E\left[\frac{K_{r,n}^{2}}{K_{n}}\right]-2p_{\alpha}(r)\E[K_{r,n}]+(p_{\alpha}(r))^{2}\E[K_{n}].
\end{equation}
On the one hand, it follows from \citet[Equations 33]{Ber(24)} with $p=1$ that
\begin{equation}
\label{MEANKN}
\E[K_{n}] =\frac{\theta}{\alpha}\left(\frac{(\alpha+\theta)^{(n)}}{(\theta)^{(n)}}-1\right)
= \E[S_{\alpha, \theta}] \left(\frac{\Gamma(n+\alpha+\theta)}{\Gamma(n+\theta)}-\frac{\Gamma(\alpha+\theta)}{\Gamma(\theta)}\right).
\end{equation}
Hence, we obtain from \eqref{MEANKN} together with the asymptotic behavior of the ratio of two Gamma functions that
\begin{eqnarray}\label{AEMEANKN}
\E[K_{n}] &=& \E[S_{\alpha, \theta}] \left(n^\alpha \left(1+ \frac{\alpha(2 \theta + \alpha -1)}{2n} + O\Big(\frac{1}{n^2}\Big) \right)-\frac{\Gamma(\alpha+\theta)}{\Gamma(\theta)}\right),  \notag\\
&=& \E[S_{\alpha, \theta}] n^\alpha \left(1  -\frac{\Gamma(\alpha+\theta)}{n^\alpha\Gamma(\theta)} +\frac{\alpha(2 \theta + \alpha -1)}{2n} + O\Big(\frac{1}{n^2}\Big) \right).
\end{eqnarray}
On the other hand, we deduce from \citet[Equations 95]{Ber(24)} with $p=1$, $2$ that
\begin{eqnarray}
\label{MEANKrN}
\E[K_{r,n}] &=& p_{\alpha}(r)\frac{n!}{(n-r)!}\left(\frac{\theta}{\alpha}\right)\frac{(\alpha+\theta)^{(n-r)}}{(\theta)^{(n)}}, \\
\E[K_{r,n}(K_{r,n}-1)] &=& (p_{\alpha}(r))^2\frac{n!}{(n-2r)!}\left(\frac{\theta}{\alpha}\right)^{(2)}\frac{(2\alpha+\theta)^{(n-2r)}}{(\theta)^{(n)}}.\label{VarKrN}
\end{eqnarray}
As before, we obtain from \eqref{MEANKrN} and \eqref{VarKrN} that
\begin{eqnarray}
\label{AEMEANKrN}
\E[K_{r,n}] 
&=& p_{\alpha}(r) \E[S_{\alpha, \theta}] n^\alpha \left(1 +\frac{C_{\alpha,\theta}(r)}{2n} + O\Big(\frac{1}{n^2}\Big) \right), \\
\E[K_{r,n}(K_{r,n}-1)] &=& (p_{\alpha}(r))^2\frac{(\alpha+\theta) \Gamma(\theta +1)}{\alpha^2 \Gamma(2\alpha + \theta)}
n^{2\alpha} \left(1 +\frac{D_{\alpha,\theta}(r)}{2n} + O\Big(\frac{1}{n^2}\Big) \right)
\label{AEVarKrN}
\end{eqnarray}
where $C_{\alpha,\theta}(r)$ and $D_{\alpha,\theta}(r)$ are given by $C_{\alpha,\theta}(r)=r(1-r)+(\alpha-r)(2 \theta + \alpha -r-1)$ and $D_{\alpha,\theta}(r)=2r(1-2r)+2(\alpha-r)(2(\theta + \alpha -r)-1)$.
Hereafter, we clearly have the decomposition
\begin{equation}
\label{MAINDECORATIO}
\E\left[\frac{K_{r,n}^{2}}{K_{n}}\right]=\E\left[\frac{K_{r,n}}{K_{n}}\right]+ \E\left[\frac{K_{r,n}^2-K_{r,n}}{K_{n}}\right].
\end{equation}
The moment equation \eqref{eq_finalmom} with $p=1$ and $q=-1$, leads to a kind of independence formula
\begin{equation}\label{RATIO1}
\E\left[\frac{K_{r,n}}{K_{n}}\right]=\frac{\E[K_{r,n}] \E[Z_{r,n}^*]}{n^\alpha}
\qquad\text{where}\qquad Z_{r,n}^*=\frac{n^\alpha}{K^{\ast}_{n-r}+1}
\end{equation}
$K^{\ast}_{n-r}$ being the number of blocks in a random partition of $[n-r]$ distributed as the Ewens-Pitman model with parameters $\alpha\in[0,1[$ and $\theta +\alpha>-\alpha$. Moreover, we also obtain from \eqref{eq_finalmom} with $p=2$ and $q=-1$ that
\begin{equation}\label{RATIO2}
\E\left[\frac{K_{r,n}^2-K_{r,n}}{K_{n}}\right]=\frac{\E[K_{r,n}^2-K_{r,n}]\E[Z_{r,n}^{**}]}{n^\alpha}
\qquad\text{where}\qquad Z_{r,n}^{**}=\frac{n^\alpha}{K^{\ast \ast}_{n-2r}+1}
\end{equation}
$K^{\ast\ast}_{n-2r}$ being the number of blocks in a random partition of $[n-2r]$ distributed as the Ewens-Pitman model with parameters $\alpha\in[0,1[$ and $\theta+2\alpha>-\alpha$. It follows from the almost sure convergence \eqref{eq:kalpha} that
\begin{equation}
\label{ASLIMKN*}
\lim_{n\rightarrow+\infty}\frac{K^{\ast}_{n-r}}{n^{\alpha}}=S_{\alpha,\theta+ \alpha}\qquad\text{a.s.}
\end{equation}
and
\begin{equation}
\label{ASLIMKN**}
\lim_{n\rightarrow+\infty}\frac{K^{\ast \ast}_{n-2r}}{n^{\alpha}}=S_{\alpha,\theta+ 2\alpha}\qquad\text{a.s.}
\end{equation}
Hence, we immediately obtain from \eqref{ASLIMKN*} and \eqref{ASLIMKN**} that
\begin{equation}
\label{ASLIMZN*}
\lim_{n\rightarrow+\infty} Z_{r,n}^{*} =S_{\alpha,\theta+ \alpha}^{-1}\qquad\text{a.s.}
\end{equation}
and
\begin{equation}
\label{ASLIMZN**}
\lim_{n\rightarrow+\infty} Z_{r,n}^{**} =S_{\alpha,\theta+ 2\alpha}^{-1}\qquad\text{a.s.}
\end{equation}
One can observe that the random variables $S_{\alpha,\theta+ \alpha}$ and $S_{\alpha,\theta+ 2\alpha}$ are positive and therefore invertible almost surely. Our goal is now to show that the convergences in \eqref{ASLIMZN*} and \eqref{ASLIMZN**} are also true in $\dL^1$. Since the parameters $\theta+\alpha>0$ and $\theta+2\alpha>0$, we deduce from Lemma \ref{lem:unif-neg-moment} and the classical de La Vall\'ee Poussin criterion that the sequence $(Z_{r,n}^{*})$ and $(Z_{r,n}^{**})$ are uniformly integrable. Hence, it follows from \eqref{ASLIMZN*} and \eqref{ASLIMZN**} together with Vitali's convergence theorem that
\begin{equation}
\label{LIMEZN*}
\lim_{n\rightarrow+\infty} \E[Z_{r,n}^*]=\E\big[S_{\alpha,\theta+ \alpha}^{-1}\big]
=\frac{\alpha \Gamma(\alpha+ \theta)}{\Gamma(\theta+1)}= \frac{1}{\E[S_{\alpha,\theta}]}
\end{equation}
and
\begin{equation}
\label{LIMEZN**}
\lim_{n\rightarrow+\infty} \E[Z_{r,n}^{**}]=\E\big[S_{\alpha,\theta+ 2\alpha}^{-1}\big]
=\frac{\alpha \Gamma(2\alpha+ \theta)}{\Gamma(\alpha+\theta+1)}.
\end{equation}
Putting together the two contributions in \eqref{MAINDECORATIO}, we obtain from \eqref{AEMEANKrN}, \eqref{AEVarKrN}, \eqref{LIMEZN*} and \eqref{LIMEZN**} that
\begin{equation}
\label{AERATIO}
\E\left[\frac{K_{r,n}^{2}}{K_{n}}\right]=(p_{\alpha}(r))^2 \E[S_{\alpha, \theta}] n^\alpha + p_{\alpha}(r)(1+o(1)).
\end{equation}
Finally, we find from \eqref{eq:main}, \eqref{AEMEANKN}, \eqref{AEMEANKrN} and \eqref{AERATIO} that
\begin{equation}
\label{AEMEAN2Qrn}
\E[Q_{r,n}^{2}]= p_{\alpha}(r)(1+o(1)) - \left(\frac{\theta}{\alpha}\right) (p_{\alpha}(r))^2.
\end{equation}
However, we clearly have from \eqref{sum_sibuya} that
\begin{equation}
\label{Rest-sibuya}
\lim_{d\rightarrow+\infty}\sum_{r=d}^\infty p_\alpha(r)=0 \qquad \text{and}
\qquad
\lim_{d\rightarrow+\infty}\sum_{r=d}^\infty (p_\alpha(r))^2=0.
\end{equation}
Therefore, \eqref{AEMEAN2Qrn} and \eqref{Rest-sibuya} lead to \eqref{toprove}, which completes the proof of Lemma \ref{L-IDCLT}. 
\end{proof}


\section{Conclusion}\label{sec3}

It has been shown that the self-normalized Ewens-Pitman process $(P_{1,n},P_{2,n},\ldots)$, once properly centered by the Sibuya distribution $p_\alpha$ and normalized by $\sqrt{K_n}$, converges in distribution to the centered Gaussian process $\mathcal{G}(\Gamma_\alpha)$. While this establishes a functional central limit theorem, several related questions remain open and would further deepen the understanding of the probabilistic structure of the self-normalized Ewens-Pitman process.

A first direction concerns large deviations in finite dimension. Let $d\geq 1$ be fixed. Is it possible to establish a large deviation principle (LDP) in $\mathbb{R}^d$, in the sense of \citet{Dem(98)}, for the random vector $(P_{1,n},\ldots,P_{d,n})?$ Even in the special case $d=1$, this question appears far from being obvious, as it requires a sharp control of rare events in a regime where both the
centering and the normalization depend intrinsically on the random number of blocks $K_n$. A positive answer would open the way to obtaining an LDP for estimators of the parameter $\alpha$, such as the estimator 
$\widehat{\alpha}_n$, see also \citet{Fen(98)} for related motivations.

A second direction concerns functional large deviations in infinite dimension. Let $\mu_n$ denote the random measure naturally associated with the self-normalized Ewens-Pitman process $(P_{1,n},P_{2,n},\ldots)$. Is it possible to prove a functional LDP, again in the sense of \citet{Dem(98)}, for the random measure $\mu_n$, in a suitable topology on the space of probability measures, or on an appropriate function space? Such a result would provide a refined description of atypical fluctuations beyond the Gaussian regime and would complement the central limit theorem proved here.

Finally, beyond purely asymptotic distributional statements, it would be of considerable interest to develop \emph{quantitative} versions of our results. In particular, one may ask whether the Gaussian approximation can be strengthened to Berry--Esseen type bounds, either in finite dimension for the random vector $(P_{1,n},\ldots,P_{d,n})$, or at the process level. Establishing explicit convergence rates, possibly depending on $\alpha$, $d$, and the growth of $K_n$, would substantially enhance the applicability of the limit theory, especially for statistical inference based on moderate sample sizes.


\appendix
\section{}\label{appa}
\renewcommand{\thesection}{A}
\renewcommand{\thesubsection}{A.\arabic{subsection}}
\renewcommand{\thesubsubsection}{A.\arabic{subsection}.\arabic{subsubsection}}
\renewcommand{\theequation}{A.\arabic{equation}}
\renewcommand{\thethm}{\thesection.\arabic{thm}}

\setcounter{subsection}{0}
\setcounter{subsubsection}{0}
\setcounter{equation}{0}
\setcounter{thm}{0}

\subsection{Sequential or generative construction of the Ewens-Pitman model}

We recall the sequential or generative construction of the Ewens-Pitman model given in \citet[Proposition 9]{Pit(95)}; see also \citet[Chapter 3]{Pit(06)}. For any $\alpha\in[0,1[$ and  $\alpha+\theta>0$, the random partition of $[n]$ distributed according to the Ewens-Pitman model, can be recursively constructed as follows: 
conditionally on the total size $K_n=k$ of the partition and on the partition subsets $\{A_1, \ldots ,A_k \}$ of corresponding sizes $(n_1, \ldots,n_k)$, the partition $[n+1]$ is an extension of $[n]$ such that the element $n+1$ is attached to subset $A_i$ for $1 \leq i \leq k$, with probability
$$
\frac{n_i- \alpha}{n+ \theta},
$$
or forms a new subset with probability
$$
\frac{\alpha k + \theta}{n+ \theta}.
$$ 
Since $n=n_1+ \cdots+n_k$, we clearly have
\begin{equation*}
    \frac{1}{n+\theta}\sum_{i=1}^{k} (n_i - \alpha) + \frac{ \alpha k + \theta}{n+\theta}= \frac{n-\alpha k + \alpha k + \theta}{n+ \theta}=1.
\end{equation*}
Hereafter, for $K_{n}\in\{1,\ldots,n\}$, let $\mathbf{N}_{n}=(N_{1,n},\ldots,N_{K_{n},n})$ be the sizes of 
the partition subsets $\{A_1, \ldots,A_{K_n}\}$. For $r=1,\ldots,n$, denote 
$$K_{r,n}=\sum_{i=1}^{K_{n}}\rI_{\{N_{i,n}=r\}}.$$ 
\citet[Proposition 9]{Pit(95)} shows that the above sequential construction leads to the joint distribution \eqref{epsm} of $\mathbf{K}_{n}=(K_{1,1},\ldots,K_{n,n})$. Equation \eqref{DEFPN} follows from the above sequential construction since
\begin{displaymath}
\mathbb{P}(K_{n+1}=K_{n}+1\,|\,K_{n})=\frac{\alpha K_{n}+ \theta}{n+\theta}.
\end{displaymath}
We also deduce Equation \eqref{DEFIPr} from the above construction
as for $r=1$,
\begin{displaymath}
\dP(\xi_{1,n+1}=1\,|\,\mathcal{F}_{n})=\frac{\alpha K_{n}+ \theta}{n+\theta}.
\end{displaymath}
Moreover, for $r=1$ and $n_1=1$,
\begin{displaymath}
 \dP(\xi_{1,n+1}=-1\,|\,\mathcal{F}_{n})=\frac{(1-\alpha)K_{1,n}}{n+\theta}.
\end{displaymath}
In addition, for all $r \geq 2$ and for $n_i=r-1$,
\begin{displaymath}
\dP(\xi_{r,n+1}=1\,|\,\mathcal{F}_{n})=\frac{(r-1-\alpha)K_{r-1,n}}{n+\theta},
\end{displaymath}
while, for $n_i=r$,
\begin{displaymath}
\dP(\xi_{r,n+1}=-1\,|\,\mathcal{F}_{n})=\frac{(r-\alpha)K_{r,n}}{n+\theta}.
\end{displaymath}
Finally, Equations \eqref{DEFK} and \eqref{DEFKr} follow from the sequential construction of the Ewens-Pitman model.

\subsection{Moment formulae for the random variable $(K_{n}, K_{r,n})$}
Assume that $\alpha\in[0,1[$ and  $\alpha+\theta>0$. We have from \eqref{epsm} and 
the probability distribution of $K_{n}$ given in \citep[Equation 3.11]{Pit(06)} that
\begin{equation}\label{eq:condep}
\dP(\mathbf{K}_{n}=(k_{1},\ldots,k_{n})\,|\,K_{n}=k)=\frac{n!}{\mathscr{C}(n,k;\alpha)}
\prod_{i=1}^{n}\left(\frac{\alpha(1-\alpha)^{(i-1)}}{i!}\right)^{k_{i}}\frac{1}{k_{i}!},
\end{equation}
where $\mathscr{C}(n,k;\alpha)$ is the generalized factorial coefficient \citep[Chapter 2]{Cha(05)} defined, for all $n \geq 0$ and for all $k \geq 0$, by
\begin{displaymath}
\mathscr{C}(n,k;\alpha)=\frac{1}{k!}\sum_{i=0}^{k}{k\choose i}(-1)^{i}(-\alpha i)^{(n)}.
\end{displaymath}
More precisely, Equation \eqref{eq:condep} is a conditional version of the Ewens-Pitman model, which is defined on the set
\begin{equation}\label{part_set}
\mathcal{K}_{n,k}=\left\{(k_{1},\ldots,k_{n})\text{ : }k_{i}\geq0,\,\sum_{i=1}^{n}k_{i}=k\text{ and }\sum_{i=1}^{n}ik_{i}=n\right\}.
\end{equation}
For any $r=1,\ldots,n$, we are going to compute the conditional falling factorial moments of $K_{r,n}$ given $K_n$, that is for all integer $p\geq 1$, $\E[(K_{r,n})_{(p)}\,|\,K_{n}=k]$ where, for any $a \in \dR$, 
$(a)_{(p)}=a(a-1)\cdots(a-p+1)$ 
with $(a)_{(0)}=1$. We obtain from \eqref{Sibuya} and \eqref{eq:condep} that
\begin{align}\label{cond_mom1}
\E[(K_{r,n})_{(p)}\,|\,K_{n}=k]
&=\frac{n!}{\mathscr{C}(n,k;\alpha)}\sum_{(k_{1},\ldots,k_{n})\in\mathcal{K}_{n,k}}(k_{r})_{(p)}\prod_{i=1}^{n}\left(p_\alpha(i)\right)^{k_{i}}\frac{1}{k_{i}!} \notag\\
&=\frac{n!}{\mathscr{C}(n,k;\alpha)} \!\! \sum_{(k_{1},\ldots,k_{n})\in\mathcal{K}_{n,k}}
\!\!\!\! \left(p_\alpha(r)\right)^{k_{r}}\frac{1}{(k_{r}-p)!}\prod_{1\leq i\neq r\leq n }\left(p_\alpha(i)\right)^{k_{i}}\frac{1}{k_{i}!} \notag\\
&=\frac{n!}{\mathscr{C}(n,k;\alpha)}\left(p_\alpha(r) \right)^{p}
\!\! \sum_{(k_{1},\ldots,k_{n-pr})\in\mathcal{K}_{n-pr,k-p}} \!\!\prod_{i=1}^{n-pr}\left(p_\alpha(i)\right)^{k_{i}}\frac{1}{k_{i}!} \notag\\
&=\frac{1}{\mathscr{C}(n,k;\alpha)}\left(p_\alpha(r) \right)^{p} (n)_{(pr)} \mathscr{C}(n-pr,k-p;\alpha),
\end{align}
thanks to \citet[Theorem 2.15 ]{Cha(05)}. Moreover, we have for all integer $p\geq 1$,
\begin{equation}
\label{stirling2}
\E[K_{r,n}^p \,|\,K_{n}=k]=\sum_{i=0}^p \left\{ 
\begin{matrix}
p \\i 
\end{matrix}
\right\}
\E[(K_{r,n})_{(i)}\,|\,K_{n}=k],
\end{equation}
where the curly brackets are the Stirling numbers of the second kind given by
$$
\left\{ 
\begin{matrix}
p \\i
\end{matrix}
\right\}= \frac{1}{i!}\sum_{j=0}^i (-1)^{i-j}  {i\choose j} j^p.
$$
Therefore, we deduce \eqref{cond_mom1} and \eqref{stirling2} that
\begin{equation}\label{cond_mom3}
\E[K_{r,n}^{p}\,|\,K_{n}=k]=\frac{1}{\mathscr{C}(n,k;\alpha)}\sum_{i=0}^p \left\{ 
\begin{matrix}
p \\i 
\end{matrix}
\right\}
\left(p_\alpha(r)\right)^{i}(n)_{(ir)}\mathscr{C}(n-ir,k-i;\alpha).
\end{equation}
Hereafter, we are able to compute the joint moments $\E[K_{r,n}^{p}K_{n}^{q}]$, for any $p\in\mathbb{N}$ and $q\in\mathbb{Z}$. More precisely, it follows from \eqref{cond_mom3} and the tower property of conditional expectation that
\begin{align*}
\E[K_{r,n}^{p}K_{n}^{q}] &= \sum_{k=1}^{n} \E[K_{r,n}^{p} K_{n}^{q}\,|\,K_{n}=k] \dP(K_n=k) \\
&=\sum_{k=1}^{n}k^{q}\frac{\left(\frac{\theta}{\alpha}\right)^{(k)}}{(\theta)^{(n)}}\mathscr{C}(n,k;\alpha)\E[K_{r,n}^{p}\,|\,K_{n}=k]\\
&=\sum_{k=1}^{n}k^{q}\frac{\left(\frac{\theta}{\alpha}\right)^{(k)}}{(\theta)^{(n)}}
\sum_{i=0}^p \left\{ 
\begin{matrix}
p \\i 
\end{matrix}
\right\}
\left(p_\alpha(r)\right)^{i}(n)_{(ir)}\mathscr{C}(n-ir,k-i;\alpha) \\
&=\sum_{i=0}^p \left\{ 
\begin{matrix}
p \\i 
\end{matrix}
\right\}
\left(p_\alpha(r)\right)^{i}(n)_{(ir)}
\sum_{k=i}^{n}k^{q}\frac{\left(\frac{\theta}{\alpha}\right)^{(k)}}{(\theta)^{(n)}}\mathscr{C}(n-ir,k-i;\alpha)\\
&=\sum_{i=0}^p \left\{ 
\begin{matrix}
p \\i 
\end{matrix}
\right\}
\left(p_\alpha(r)\right)^{i}(n)_{(ir)}
\sum_{\ell=0}^{n-i}(\ell+i)^{q}\frac{\left(\frac{\theta}{\alpha}\right)^{(\ell+i)}}{(\theta)^{(n)}}
\mathscr{C}(n-ir,\ell;\alpha)\\
&=\sum_{i=0}^p \left\{ 
\begin{matrix}
p \\i 
\end{matrix}
\right\}
\left(p_\alpha(r)\right)^{i}(n)_{(ir)}
\left(\frac{\theta}{\alpha}\right)^{(i)}
\frac{\left(\theta+i \alpha \right)^{(n-ir)}}{(\theta)^{(n)}}
\mathscr{E}(n-ir,q)
\end{align*}
where
\begin{equation*}
\mathscr{E}(n-ir,q)=\sum_{\ell=1}^{n-ir}(\ell+i)^{q}\frac{\left(\frac{\theta}{\alpha}+i\right)^{(\ell)}}{(\theta+i \alpha)^{(n-ir)}}
\mathscr{C}(n-ir,\ell;\alpha).
\end{equation*}
One can observe that
\begin{displaymath}
\mathscr{E}(n-ir,q)=\E[(K^{\ast}_{n-ir}+i)^{q}]
\end{displaymath}
where $K^{\ast}_{n-ir}$ is for the number of blocks of the random partition of $[n-ir]=\{1,\ldots,n-ir\}$ distributed according to the Ewens-Pitman model with parameter $\alpha\in[0,1[$ and $\theta+i\alpha>-\alpha$.
Finally, we can conclude that
\begin{equation}\label{eq_finalmom}
\E[K_{r,n}^{p}K_{n}^{q}]
= \sum_{i=0}^p \left\{ 
\begin{matrix}
p \\i 
\end{matrix}
\right\}
\left(p_\alpha(r)\right)^{i}(n)_{(ir)}
\left(\frac{\theta}{\alpha}\right)^{(i)}
\frac{\left(\theta+i \alpha \right)^{(n-ir)}}{(\theta)^{(n)}}
\E[(K^{\ast}_{n-ir}+i)^{q}].
\end{equation}

\subsection{Negative moments of $K_{n}$}

Our last lemma, which may have its own interest, provides a uniform bound on the negative moments of $K_{n}$.

\begin{lem}\label{lem:unif-neg-moment}
Let $\alpha\in]0,1[$ and assume that $\theta>0$. Then, for any $q\in]0,1\!+\!\theta/\alpha[$, we have
\begin{equation}
\label{NEGMKN}
\sup_{n\ge1}\E\left[\left(\frac{K_n}{n^\alpha}\right)^{-q}\right]<+\infty .
\end{equation}
\end{lem}

\begin{proof}
For any $q\in]0,1\!+\!\theta/\alpha[$, we have according to \citet[Formula (3.45)]{Pit(06)} that
\begin{equation}
\label{PKSIB}
\E[K_n^{-q}]
= \sum_{k=1}^n \frac{1}{k^{q}}\, \dP(K_n=k)=\sum_{k=1}^n \frac{1}{k^{q}}
\frac{n!}{k!}\,\frac{\left(\frac{\theta}{\alpha}\right)^{(k)}}{(\theta)^{(n)}}\,
\dP(S_k=n)
\end{equation}
where $S_k=X_1+\cdots+X_k$, with the random variables $X_{i}$'s being independent and identically distributed sharing the same Sibuya distribution $p_{\alpha}$. Moreover, we also have from \citet[Formula (3.44)]{Pit(06)} that for all $n \geq 1$ and for all $1\leq k \leq n$,
\begin{align} 
\dP(S_k=n)&=\frac{1}{n!}\sum_{i=1}^k (-1)^{i}\binom{k}{i}\,(-i \alpha)^{(n)}=\frac{1}{n!}\sum_{i=1}^k (-1)^{n+i}\binom{k}{i}\,(i\alpha )_{(n)}  \notag \\
 &=\frac{1}{n!} \sum_{i=1, i \alpha \notin\dN}^k (-1)^{n+i}\binom{k}{i}\,\frac{\Gamma(i\alpha +1)}{\Gamma(i\alpha+1-n) }. 
\label{PKSIB1}
\end{align}
However, we know from Euler's reflection formula that for all $i \geq 1$ such that 
$i \alpha \notin\dN$, 
\begin{equation*}
\Gamma(i\alpha+1-n)\Gamma(n-i\alpha)=\frac{\pi}{\sin(\pi(n-i\alpha))}.
\end{equation*}
Hence, we obtain from \eqref{PKSIB1} that
\begin{equation} 
\dP(S_k=n)= \sum_{i=1, i \alpha \notin\dN}^k (-1)^{i+1}\binom{k}{i}\,\frac{\sin(i \alpha \pi)\Gamma(i\alpha +1)}{\pi} \frac{\Gamma(n-i\alpha)}{\Gamma(n+1)}. 
\label{PKSIB2}
\end{equation}
Furthermore, it follows from Wendel’s inequality for the ratio of gamma functions that there exists a constant $C_\alpha >0$ such that
for all $n \geq 1$,  
$$
\frac{\Gamma(n-i\alpha)}{\Gamma(n+1)} \leq  \frac{C_\alpha}{n^{1+i \alpha}}.
$$
Consequently, \eqref{PKSIB2} immediately leads to 
\begin{align} 
\dP(S_k=n)&\leq \frac{C_\alpha}{\pi}\sum_{i=1}^k \binom{k}{i}\,\frac{\Gamma(i\alpha +1)}{n^{1+i \alpha} } \leq \frac{C_\alpha}{\pi}\sum_{i=1}^k \frac{k^i}{i!}\,\frac{\Gamma(i\alpha +1)}{n^{1+i \alpha} } \notag \\
&\leq \frac{C_\alpha}{n \pi}\sum_{i=1}^k \frac{\Gamma(i\alpha +1)}{i!}\,\left(\frac{k}{n^{\alpha} }\right)^i.
\label{PKSIB3}
\end{align}
In addition, since $\alpha\in]0,1[$, $\Gamma(i\alpha+1)/i!$ decays super-exponentially in $i$. In particular, there exists a constant $D_\alpha>0$ such that for all $i \geq 1$,
\begin{equation}
\label{PKSIB4}
\frac{\Gamma(i\alpha +1)}{i!}\le \frac{D_\alpha}{i^{(1-\alpha)i}}.
\end{equation}
Consequently, for all $1 \leq k \leq n$ such that $k\le n^\alpha$, we have $(k/n^\alpha)^i\le 1$ and we obtain from \eqref{PKSIB3} and \eqref{PKSIB4}
that 
\begin{equation}
\label{PKSIB5}
\dP(S_k=n)\le \frac{C_\alpha D_\alpha}{n \pi}\sum_{i=1}^\infty \frac{1}{{i^{(1-\alpha)i}}}\left(\frac{k}{n^\alpha}\right)^i\le 
 \frac{E_\alpha\,k}{n^{1+\alpha}},
\end{equation}
for some constant $E_\alpha>0$, by virtue of the elementary fact that
$$
\sum_{i=1}^\infty \frac{1}{{i^{(1-\alpha)i}}}< \infty.
$$
Hereafter, we focus our attention on the prefactor into \eqref{PKSIB}.
We obtain once again from Wendel’s inequality that there exist two constants $C_\theta >0$ and $D_{\alpha, \theta} >0$ such that
for all $n \geq 1$ and for all $k \geq 1$,
\begin{displaymath}
\frac{n!}{(\theta)^{(n)}}=\frac{\Gamma(n+1)\Gamma(\theta)}{\Gamma(n+\theta)}
\le C_\theta\,n^{1-\theta}
\end{displaymath}
and
\begin{displaymath}
\frac{\big(\frac{\theta}{\alpha}\big)^{(k)}}{k!}
=\frac{\Gamma(\frac{\theta}{\alpha}+k)}{\Gamma(\frac{\theta}{\alpha})\Gamma(k+1)}
\le D_{\alpha, \theta}\,k^{\theta/\alpha-1}.
\end{displaymath}
Therefore, it follows from \eqref{PKSIB} and \eqref{PKSIB5} that there exists a constant $E_{\alpha, \theta} >0$ such
for all $1\leq k \leq n$ where $k\le n^\alpha$,
\begin{equation}
\label{PKSIB6}
\dP(K_n=k)\le \frac{E_{\alpha, \theta}\,   k^{\theta/\alpha}}{n^{\alpha+\theta}}.
\end{equation}
Consequently, by splitting \eqref{PKSIB} into two terms, we obtain from \eqref{PKSIB6} with $\kappa_n=\lfloor n^\alpha\rfloor$ that
\begin{align}
\E\left[\left(\frac{K_n}{n^\alpha}\right)^{-q}\right]
&= n^{\alpha q} \sum_{k=1}^{\kappa_n} \frac{1}{k^{q}}\, \dP(K_n=k)
+ n^{\alpha q} \sum_{k=\kappa_n +1}^{n} \frac{1}{k^{q}}\, \dP(K_n=k), \notag \\
& \leq n^{\alpha q} \sum_{k=1}^{\kappa_n} \frac{1}{k^{q}}\, \frac{E_{\alpha, \theta}\,   k^{\theta/\alpha}}{n^{\alpha+\theta}}
+ n^{\alpha q}n^{-\alpha q} \sum_{k=\kappa_n +1}^{n}  \dP(K_n=k),
 \notag \\
&\leq 1+ \frac{E_{\alpha, \theta}}{n^{\alpha+\theta - \alpha q}} \sum_{k=1}^{\kappa_n}  k^{\theta/\alpha -q}\,.
\label{PKSIB7}
\end{align}
Since $q<1+\theta/\alpha$, we have $q_{\alpha,\theta}=\theta/\alpha+1-q>0$, which implies that
\begin{displaymath}
\sum_{k=1}^{\kappa_n}  k^{\theta/\alpha -q}\le \frac{1}{q_{\alpha,\theta}}\kappa_n^{q_{\alpha,\theta}}
\leq \frac{1}{q_{\alpha,\theta}}\,n^{\theta+\alpha-\alpha q}.
\end{displaymath}
Finally, by plugging this inequality into \eqref{PKSIB7}, we immediately obtain \eqref{NEGMKN}, which completes the proof of Lemma 
\ref{lem:unif-neg-moment}.
\end{proof}


\section*{Acknowledgement}

Stefano Favaro is also affiliated to IMATI-CNR ``Enrico  Magenes" (Milan, Italy). Stefano Favaro acknowledge the financial support from the Italian Ministry of Education, University and Research (MIUR), ``Dipartimenti di Eccellenza" grant 2023-2027. 


\end{document}